\documentclass[12pt,a4paper,english]{smfart}

\usepackage{fullpage}

\usepackage[french]{babel}
\usepackage[utf8x]{inputenc}
\usepackage{smfthm,enumerate}
\usepackage{amssymb,url,xspace,euscript}
\usepackage{enumerate}
\usepackage[arrow,matrix]{xy}
\xyoption{all}

\xyoption{frame}

\usepackage{color}
\usepackage[colorlinks,linkcolor=blue,urlcolor=blue,linktocpage,dvips]{hyperref}
\usepackage{amsmath,amscd,amssymb}

\usepackage{calrsfs}

\newcommand{\finpreuve}{\mbox{} \hfill \mbox{$\Box$}}

\def\K{{\rm K}}
\def\L{{\rm L}}
\def\F{{\rm F}}
\def\k{{\rm k}}

\def\Q{{\mathbb Q}}
\def\Z{{\mathbb Z}}
\def\fq{{\mathbb F}}
\def\Com{{\mathbb C}}
\def\Reel{{\mathbb R}}

\def\N{{\rm N}}

\def\d{{\rm disc}}

\def\CC{{\alpha}}

\def\H{{\mathcal H}}
\def\G{{\mathcal G}}

\def\X{{\mathcal X}}

\def\O{{\mathcal O}}

\def\U{{\mathcal U}}
\def\V{{\mathcal V}}

\def\T{{\mathcal T}}

\def\LL{{\mathcal L}}

\def\Mo{{\mathbb M}}

\def\p{{\mathfrak p}}
\def\P{{\mathfrak P}}

\def\q{{\mathfrak q}}
\def\mm{{\mathfrak m}}
\def\mmm{{\overline{\mathfrak m}}}

\def\I{{\rm I}}

\def\Rd{{\rm Rd}}
\def\deg{{\rm deg}}

\def\Im{{\rm Im}}
\def\Gal{{\rm Gal}}

\def\grad{{\rm Grad}}
\def\Cl{{\rm Cl}}

\def\GST{{\G_S^T}}

\def\Rg{{\rm Reg}}
\def\A{{\rm A}}
\def\X{\mathrm{X}}

\def\Minf{{\underline{\Mo}}}
\def\Msup{{\overline{\Mo}}}

\def\1{{\bf 1}}

\def\FF#1#2{{\displaystyle{ \left( \frac{#1}{#2} \right) }} }

\def\a#1{\frac{\log #1}{\sqrt{#1}-1}}

\newenvironment{Question}{\begin{enonce}{Question}}{\end{enonce}}

\begin{document}

\title{On the invariant factors of class groups in towers of number fields}

\date{\today}

\author{Farshid Hajir}
\address{Department of Mathematics \& Statistics, University of Massachusetts, Amherst MA 01003, USA.}
\author{Christian Maire}
\address{Laboratoire de Math\'ematiques, UMR 6623 Universit\'e de Franche-Comté et CNRS, 16 route de Gray, 25030 Besan\c con, France}
\thanks{\emph{Acknowledgements.} The second author would like to thank UMass Amherst for its hospitality during several visits, as well as the  R\'egion Franche-Comt\'e for making his travel possible.  He also thanks NTU Singapore for providing a stimulating research atmosphere on a visit there.}

\begin{abstract}
 For a finite abelian $p$-group $A$ of rank $d=\dim A/pA$, let $\Mo_A := \log_p |A|^{1/d}$ be its
 \emph{(logarithmic) mean exponent}.  We study the behavior of the mean exponent of $p$-class groups in pro-$p$ towers $\L/K$ of number fields.  Via a combination of results from analytic and algebraic number theory, we construct infinite tamely
 ramified pro-$p$ towers in which the mean exponent of $p$-class groups remains bounded.  Several explicit
 examples are given with $p=2$.   Turning to group theory, we introduce an invariant $\Minf(G)$ attached to a finitely generated pro-$p$ group $G$; when $G=\Gal(\L/\K)$, where $\L$ is the Hilbert $p$-class field tower of a number field $K$, $\Minf(G)$ measures the asymptotic behavior of the mean exponent of $p$-class groups inside $\L/\K$.  We compare and contrast the behavior of this invariant in analytic versus non-analytic groups.  We exploit the interplay of group-theoretical and number-theoretical perspectives on this invariant and explore some open questions that arise as a result, which may be of independent interest in group theory.
 
 \end{abstract}

\maketitle

\tableofcontents

A few hundred years after its definition, the ideal class group
 continues to be one of the most mysterious
objects in number theory. One early lesson, going back to Gauss, was
that it is advantageous to study the $p$-Sylow subgroup
of the class group one prime $p$ at a time.  The variation of $p$-class
groups in pro-$p$ towers of number fields is perhaps the area that has had the
most success, thanks to the pioneering work of Iwasawa. Indeed, his
insights uncovered a very rich algebraic structure in the behavior of
$p$-class groups in layers of a $\Z_p$-extension.  In particular, the
growth of the generator rank of these $p$-class groups is governed by
the  invariants $\mu, \lambda, \nu$ which derive from the 
structure of the associated Iwasawa module.  These ideas have been
extended to a much broader context of extensions with more general 
$p$-adic analytic groups, including non-abelian ones (see, for example,
Harris
\cite{Harris1}, Venjakob \cite{Venjakob},
Coates-Schneider-Sujatha \cite{CSJ}, Perbet \cite{Perbet}, to cite only a few 
authors). 

\

In this article, we consider the variation of the invariant factors of
$p$-class groups, focusing in particular on a notion we call the
``mean exponent'' in towers of $p$-extensions of number fields.  A
recurring theme is comparing and contrasting the tame case versus the
analytic case; indeed, the Fontaine-Mazur conjecture \cite[Conjecture 5a]{FM}
has influenced
and motivated the questions we explore here.

\

First, let's define the average or mean
exponent.  Suppose a non-trivial finite $p$-group $A$ has elementary
divisors $p^{a_1}, \ldots, p^{a_d}$ listed in decreasing order, in
other words
$$ A = \Z/p^{a_1} \times \ldots \times \Z/p^{a_{d}},
\qquad a_1 \geq a_2 \geq \ldots \geq a_{d}\geq 1,
$$
where $d$ is the $p$-rank of $A$.
We then define the \emph{(logarithmic) mean exponent} of $A$ to be 
$$\Mo_A:=\frac{a_1 + a_2 + \cdots + a_{d}}{d}=  \log_p|A|^{1/d}=\frac{\log_p |A|}{d},$$
where $\log_p(a) = \log(a)/\log(p)$ is the base-$p$ logarithm.  Thus,
the mean exponent is a \emph{normalized}  measure of the size
of the group as compared to its rank.
Note that for a non-trivial $p$-group $A$, we always have 
$1\leq \Mo_A \leq \log_p |A|$, the minimum value occurring in the
case where $A$ is an elementary abelian $p$-group and the maximum value
occurring in the case of cyclic $A$.  Note also that
$\exp(A)=p^{a_1}$ is the exponent of $A$.
The mean exponent of the trivial group
is defined to be $0$.

\

For a number field $\K$, we denote by $\A(\K)$ its  $p$-class group, and we put
$$\Mo(\K,p):=\Mo(\K)= \Mo_{\A(\K)}$$
to be the ``mean exponent'' of the $p$-class group of $K$.

\

Second, let us introduce towers with restricted ramification.
Let $\K$ be a number field, $p$ a rational prime number, and $S$, $T$ a
disjoint pair of finite sets of places of $\K$.  Inside a fixed
algebraic closure of $\K$, consider the compositum $\K_S^T$ of all
finite Galois extensions of $\K$ of $p$-power degree unramified outside
$S$ and in which all the places of $T$ split completely.  We call
$\K_S^T$ the maximal unramified-outside-$S$ and $T$-split $p$-extension
of $\K$, and put $\G_S^T=\G_S^T(\K,p) =\Gal(\K_S^T/\K)$ for its Galois group
over $\K$.  If there are no places dividing $p$ in $S$, which we
abbreviate as $(S,p)=1$ and call the tame case, the structure of the
groups $\G_S^T$ is rather mysterious. In particular, it's already
difficult to determine in any given case whether $\G_S^T$ is finite or
not.  On the other hand, if $S$ contains all the primes of $\K$
dividing $p$ (the wild case), then the knowledge of $\Z_p$-extensions
of $\K$, which give infinite abelian quotients of $\G_S^{\emptyset}$, goes
quite far in revealing the structure of the latter group.  By
contrast, in the tame case, $\G_S^T$ is \emph{FAb}, meaning its
subgroups of finite index have finite abelianization, so in particular there are no surjections
to $\Z_p$.  This is a
manifestation of a broader philosophy of Fontaine and Mazur \cite{FM} that maintains that
``geometric'' $p$-adic Galois representations with infinite image are
always wildly ramified.  The dichotomy of the wild and tame cases is
punctuated by the expectation that when $(S,p)=1$, $\G_S^T$ has no
infinite $p$-adic analytic quotients.

\

To illustrate the key ideas, let us fix $p$, and consider a number field $\K$ with infinite Hilbert
$p$-class field, i.e.~ $\G_\emptyset^\emptyset(\K)$ is infinite.  Let us fix an infinite Galois
extension $\L/\K$ with $\K \subset \L \subseteq K_\emptyset^\emptyset$.  
We are primarily interested in estimating  $\exp(\A({\K_n}))$, for $(\K_n)$ a nested sequence
inside $\L$, but finding this
difficult, we also study $(\Mo(\K_n))$, i.e. 
the variation of the mean exponent of $p$-class groups in the tower $\L/\K$. In particular, for each natural number $n$,
 we define $$\Mo_n(\L/\K)= \min_{[\K':\K]=p^n} \ \Mo(\K'),$$
 where the minimum is taken over all extensions $\K'/\K$ of degree $p^n$ with $ \K' \subset\L$.  We then put $$
\Minf(\L/\K)= \liminf_n \ \Mo_n(\L/\K),
$$
which we call the \emph{asymptotic mean exponent} of the tower.  This quantity is well-defined but could \emph{a priori} be $\infty$.  

\

Let's note right away that these asymptotic invariants can be defined
purely in a group-theoretical context, as follows.  Say $\G$ is an
infinite finitely generated \emph{FAb} pro-$p$ group.  For each $n$,
we put
$$\Mo_n(\G) = \min_{[\G:\U]=p^n} \Mo_{\U^{\mathrm{ab}}}$$ where the
minimum is taken over the open subgroups of index $p^n$. We then put
$$
\Minf(\G)=\liminf_n \Mo_n(\G)
$$ for the asymptotic mean exponent of $\G$.  It's clear that if $\G =
\Gal(\L/\K)$, with $\L=\K_\emptyset^\emptyset$, then
$\Minf(\G)=\Minf(\L/\K)$.  Let's also note that we immediately have
the estimate $1 \leq \Minf(\G)$ but a general upper bound would seem
to be elusive.

\

Some of our results in this paper give bounds for
$\Minf(\L/\K)$ for certain kinds of tame extensions $\L/\K$.  In
particular, we draw upon a relationship between the number of primes
that split in $\L/\K$ and the asymptotic mean exponent of the tower.
Thus for finitely generated infinite \emph{FAb} $\G$ which are
realizable as the Galois group of the Hilbert $p$-class tower of
number fields, we can bound $\Minf(\G)$ from above.  These estimates
could be of
interest in relation to the following question: is every finitely
generated \emph{FAb} pro-$p$ group realizable as
$\Gal(\K_\emptyset^\emptyset /\K)$ for some number field $\K$?  In
particular, it raises the purely group-theoretical question of whether
there exists an infinite finitely generated \emph{FAb} pro-$p$ group
$\G$ for which $\Minf(\G)=\infty$?  Note that Ozaki \cite{Ozaki} has shown that
for any  \emph{finite} $p$-group $\G$, there exists a number field
$K$ such that $\G$ is isomorphic to $\Gal(\K_\emptyset^\emptyset/\K))$.

\

The following theorem summarizes some of the key results  in this paper.

\begin{theo}\label{intro}
\begin{enumerate}
\item
Suppose $S$ is a finite set of primes of a number field 
$\K$ with $(S,p)=1$ such that $\G=\Gal(\K_S^\emptyset/\K)$ is infinite.
  Then there exists a constant $C>0$ such that for all open subgroups $\U \subset \G$, $\Mo_{\U^{\mathrm{ab}}} \leq C [\G:\U]$.

\item With $K,S,\G$ as above, suppose $\G$ is mild (for example this is the case if $K,S$ satisfy the condition of Labute \cite[Theorem 1.6]{labute}, and see also Schmidt \cite{schmidt}).  Then for all $\varepsilon >0$, there exist a constant $C'>0$ and a nested sequence of open subgroups $\U_i$ forming an open neighborhood of $\G$ such that 
 $\Mo_{\U_i^{\mathrm{ab}}} \leq C' [\G:\U_i]/(\log [\G:\U_i])^{2-\varepsilon}$.
 
 \item There exist infinitely many pairwise disjoint number fields $\K$  with infinite $p$-class field tower $\K_\emptyset^\emptyset/\K$ but finite asymptotic mean exponent, i.e. $\Minf(\Gal(\K_\emptyset^\emptyset/\K)) \neq \infty$.
\end{enumerate}
\end{theo}

\

The first two parts of the theorem come relatively easily from standard techniques; they are proved in Proposition \ref{theo-1} and Theorem \ref{theo-2}, respectively.  To illustrate the third part, which is proved in \S \ref{pc}, consider the following concrete arithmetic example.  Namely, fix $p=2$ and let $\K$ be the following compositum of quadratic fields:
$$
\K=\Q(\sqrt{130356633908760178920},\sqrt{-80285321329764931}).
$$
  Let $\L=\K_\emptyset^\emptyset$.  Then $\L/\K$ is infinite and 
$$
\Minf(\L/\K) \leq 8.858.
$$
The details of the construction are given below in \S3, but here, let us unpack what this example means concretely.  Namely, the assertion is that there exists a tower $\K =\K_1 \subset \K_2 \subset \ldots$ inside $\L$ such that for all $n$, the $2$-class group of $\K_n$ has mean exponent at most $8.858$, so in particular, there is always at least one elementary divisor of size at most $2^8$ all the way up the tower.  We should note that the construction of the tower guarantees that the rank of the $2$-class groups tends to infinity, so the fact that the mean exponent remains below 9 entails that the number of elementary divisors of size at most $2^8$ becomes arbitrarily large as we climb the tower.  

\

We would like to contrast the third part of the theorem with the generic behavior of the mean exponent of open neighborhoods in analytic pro-$p$ groups.  Namely, if $G$ is a uniform pro-$p$ group of dimension $d$ and $\U$ runs over the $p$-central series of $G$, we have $$\Mo_{\U^{\mathrm{ab}}} \geq \frac{1}{d} \log[G:\U],$$
hence it tends to infinity; see Corollary \ref{analytic}.

\

The principle behind the above example and others we construct is as follows.  We use genus theory to create towers in which the $p$-rank grows linearly with the degree; this is achieved by first having a tower in which many primes split and then composing that tower with a degree $p$ Galois extension the same primes ramify.  The linear growth of the rank of the $p$-class group when combined with upper bounds on the class number coming from the generalized Brauer-Siegel theorem of Tsfasman-Vladut gives us the desired upper bound on $\Minf$.  

\

In the more classical case of Iwasawa theory, i.e. in wild towers, there is an algebraic theory of the invariants $\mu,\lambda,\nu$ associated with the Iwasawa module, and having linear growth in the rank is tantamount to having $\mu>0$.  It is curious that in that context also, the phenomenon of linear rank growth appears to be related to having a large set of primes splitting in the tower (see Iwasawa \cite{Iw}).  In a forthcoming work, we will study this relationship further.

\

The paper is organized as follows.  In \S 1, we recall some background, including the work of Tsfasman-Valduts extending the Brauer-Siegel Theorem and some basic results from genus theory.   In \S 2, we begin by giving a sketch of our main construction for unramified towers, then enlarge the scope of our study by introducing class groups that classify extensions with prescribed splitting and (tame) ramification. In \S 3, we work out a number of examples in detail, demonstrating how the 
exact asymptotic formula of Tsfasman-Vladut can be exploited to improve the bounds on the mean
exponent. In \S4, we reflect on the relationship between linear growth for $p$-ranks of class groups and the existence of many primes in the tower that split (almost) completely, together with the implication of these considerations for bounding the asymptotic mean exponent in infinite tame extensions.  We turn in \S 5 from number theory to considerations of the asymptotic mean exponent for pro-$p$ groups in general.  Finally, in \S 6, we take stock of some questions for further study in group theory, as well as in number theory, that are raised by the considerations of this paper.







\section*{Some Notation}

We fix a prime number  $p$.  Let
$\K$ be a number field of signature $(r_1,r_2)$. Denote by  $\d(\K)$ the  discriminant of $\K$, by  $\Rd_\K:=|\d(\K)|^{1/[\K:\Q]}$ its root discriminant 
and by $g=g_\K=\log \sqrt{|\d(\K)|}$ its genus.
Let  $\Rg_\K$ be the regulator of $\K$. The quantity $\delta_\K$ is $1$ is $\K$ contains the $p$-roots of the unity, $0$ otherwise.

\medskip

Now let  $S$ and $T$ be two disjoints finite sets of places of $\K$. Let  $\K_S^T$ be the maximal unramified outside $S$ and $T$-split
$p$-extension 
of $\K$ (with the convention that for  $p=2$ all real places stay real); put $\GST = \Gal(\K_S^T/\K)$. It is well-known that 
the pro-$p$-group $\GST$ is finitely presented (see for example \cite{koch} or \cite{gras}): the quantities 
$$d(\GST)= \dim_{\fq_p} H^1(\GST,\fq_p)= d_p H^1(\GST,\fq_p)$$
and $$r(\GST)= \dim_{\fq_p} H^2(\GST,\fq_p)= d_p H^2(\GST,\fq_p)$$
are finite. Put   $\A_S^T:= \GST^{ab}$ 
 the maximal abelian quotient of
$\GST$, which corresponds by Class Field Theory to the maximal abelian $S$-ramified ({\it i.e.} unramified outside $S$) and $T$-split extension of $\K$.
For $S=T=\emptyset$, $\GST$ corresponds to   the Galois group of   the Hilbert $p$-Class Field Tower of $\K$ and $\A=\Cl$ to its  $p$-Class group. If 
 $S$ is prime to $p$,  the pro-$p$-group  $\GST$ is \emph{FAb}: every open subgroups of $\GST$ has finite abelianization.

\section{Background}
\subsection{Brauer-Siegel and  Tsfasman-Vladut Theorems}\label{sectionTV}

We recall first some results due to Tsfasman and Vladut  \cite{TV} generalizing the Theorem of Brauer-Siegel. 
Throughout this work, we will use the Tsfasman-Valdut context of asymptotically
exact extensions.

\medskip A sequence  $(\K_n)$ of number fields is called a tower if for all $n$, $\K_n \subsetneq \K_{n+1}$ so in particular $[\K_n:\K] \rightarrow \infty$ with $n$. 
\medskip

Let $\L/\K$ be an infinite extension of a  number field $\K$ and  let  $(\K_n)$ be a tower  
in $\L/\K$ with limit $\L$, i.e.~each  $\K_n$ is a finite extension of $\K$ contained in $\L$  and  $\displaystyle{\bigcup_n \K_n =\L}$.
Put  $$g_n=g_{\K_n}= \log(\sqrt{|\d({\K_n})|}).$$

For every prime number $\ell$ and power $q:=\ell^m$ of $\ell$, let us consider   the quantity   $$\phi_q=\lim_n \frac{N_n(q)}{g_n},$$
 where $N_n(q)=\# \{\text{prime ideal } \q \subset \O_{\K_n}, \  \#\O_{\K_n}/\q=q\}$.  We also put
 $$\phi_\Reel=\lim_n \frac{r_1(\K_n)}{g_n}, \ \ \phi_\Com=\lim_n \frac{r_2(\K_n)}{g_n}.$$

As  the sequence  $(\K_n)_n$ is a tower, all the limits exist  and  depend only  on $\L/\K$.  In the terminology of \cite{TV}, the sequence $(\K_n)_n$  is said  to be {\it asymptotically exact}. It is called
\emph{asymptotically good} if $\phi_q > 0$ for some $q$ where $q$ is either a prime power or belongs to $\{\Reel, \Com\}$.  In this paper, we will mostly be interested in examples where $\phi_\Com>0$.
Deeply ramified wild extensions (such as 
$\Z_p$-extensions) are asymptotically bad.  By contrast, 
assuming  $\G_S^\emptyset(\K,p)$ is infinite for some finite $S$ with $(S,p)=1$, any tower inside
$\K_S^\emptyset/\K$ is asymptotically good.  More generally, even if $(S,p)\neq 1$ but $(\K_n)$ is
a tower in which the $N$th
higher ramification groups all vanish for some fixed $N$, then the tower is asymptotically good (see
\cite{HM-IMRN}).

\medskip

In \cite{TV},  Tsfasman and Vladut have studied the behavior of the quantity $\log(\Rg_n \cdot h_n)/g_n$ along a tower $(K_n)$ with limit $\L/\K$, 
where $\Rg_n$
is the regulator of $\K_n$
and  $h_n$ its class number. They conjecture that the quantity $$B(\L/\K)=\lim_n \frac{ \log(\Rg_n h_n)}{g_n}$$ is well-defined and prove the following theorem.

\begin{theo}[Tsfasman-Vladut, \cite{TV}] \label{TV}
 \begin{enumerate}
 \item
Assuming GRH, the limit $B(\L/\K)$ exists and depends only $\L/\K$, not on the choice of  tower $(\K_n)_n$ with limit $L$.
Moreover one has the equality: $$B(\L/\K)= 1 + \sum_q \phi_q \log \frac{q}{q-1} - \phi_\Reel \log 2 - \phi_\Com \log 2\pi.$$
Without assuming GRH, one has the same conclusion if  the tower of  number fields $(\K_n)$ is  Galois relative to $\K$.
\item Assuming GRH,  $B(\L/\K)\leq 1.0939$ for all $\L/\K$.  If $\K$ is totally imaginary, then $B(\L/\K) \leq 1.0765$.  Without assuming GRH, one has
$B(\L/\K) \leq 1.1589$. 
\end{enumerate}
\end{theo}


\subsection{On the  $p$-$S$-$T$ towers}


Comprehensive references for the study of extensions with restricted ramification include Koch \cite{koch}, Gras \cite{gras} and Neukirch-Schmidt-Wingberg \cite{NSW}.We give only a quick sketch of some well-known facts, and refer the reader to these books which contain much more background and detail.

\

Let  $\K$ be a number field and let  $S$ and  $T$ be two finite sets of  places of $\K$ with $S\cap T= \emptyset$. We assume that $(S,p)=1$. 
We recall that the pro-$p$-group $\GST$ is \emph{FAb} and that 
the $p$-rank  $d_p \GST$ of $\GST$ can be computed thanks to Class Field Theory. In particular, one has (see e.g.~\cite{gras}, Chapter I \S 4, Theorem 4.6):

\begin{prop}\label{proposition1.2} With notation as above, we have
 $$d_p \GST = d_p \A_S^T  \geq |S| -\big(r_1(\K)+r_2(\K) + |T|-\delta_\K\big).$$
\end{prop}

{\it A priori}, the pro-$p$-group  $\GST$ may be finite or not. A criterion for its infinitude can be obtained as a consequence 
of the Theorem of  Golod-Shafarevich; the following is their result, in the improved version due to Gasch\"utz and Vinberg.
\begin{theo}[Golod-Shafarevich] \label{GS}
 If a non-trivial pro-$p$-group $\G$ is finite then its generator and relation ranks satisfy the following inequality: $\displaystyle{r(\G) \geq \frac{d(\G)^2}{4}}$.
\end{theo}

The following classical theorem of Shafarevich on the Euler characteristic of  $\GST$ is of fundamental importance in this theory (see for example \cite{gras}): 
\begin{prop}\label{boundrelations} Assuming as above that $(S,p)=1$, we have
 $$0\leq r(\GST)-d(\GST) \leq r_1+r_2 -1 + \delta_{S,p} + |T|.$$
\end{prop}
These two last propositions together imply that if $S$ is large in comparison to  the size of $T$, then  $\GST$ is infinite, giving rise to the so-called Golod-Shafarevich criterion.
This criterion can be made effective by using genus theory  (cf \cite{Maire-Bx}) to construct number fields with class group of large $p$-rank. The following is a standard result from genus theory.

\begin{theo}\label{genre}
Let $\K/\k$ be a cyclic extension of degree $p$.
Then $$d_p \A_\K \geq \rho - 1 - \big(r_1(\k)+r_2(\k)-1 + \delta_k\big),$$
 where $\delta_\k=1$ if $\k$ contains the  $p$-roots of the unity and where
 $\rho$ is the number of ramified places of $\k$ in   $\K/\k$ (eventually archimedean places).
\end{theo}

It is possible to obtain a $T$-split version of Genus Theory and then one can show \cite{Maire-PMB}:

\begin{theo}\label{theocritere}
 Let $\K/\k$ be a cyclic extension of degree $p$. 
 Assume that $$\rho +i_T \geq 3+ r_1(\k)+r_2(\k)+|T(\k)|-1+\delta_{\k}  + 2\sqrt{ r_1(\K)+r_2(\K) +|T(\K)| +\delta_{\K} }$$
  where  $\rho$
 is the number of places  ramified in  $\K/\k$ (eventually the archimedean places) and where $i_T$ is the number of places of $T$ inert in $\K/\k$. Then $\G^T:=\G_\emptyset^T$ is infinite.
\end{theo}

\begin{coro}\label{critere}
Let $\K/\Q$ be a real quadratic field and let  $T$ be a finite set of odd primes of $\Q$. 
Put $T_{\mathrm{dec}}=\{ \ell \in T, \ \ell {\rm \ splits \ in \ }  \K/\Q\}$.
If $$\rho \geq 4 + |T_{\mathrm{dec}}|+2\sqrt{3+|T|},$$ where  $\rho$ is the number of primes
 not in $T$ that are ramified in  $\K/\Q$, 
then the group $\G^T$ is infinite.
\end{coro}

\begin{proof}
We simply remark  that a prime of  $T$ which is not splitting in $\K/\Q$
 is inert or ramified and then apply Theorem \ref{theocritere}.
\end{proof}

\section{Towers with bounded mean exponent}

\subsection{The Principal Construction}\label{pc}

In this subsection, we sketch the key idea for the construction of towers with $p$-class groups of bounded mean exponent, in the simpler case of unramified extensions, and in particular, we prove Part 3 of Theorem \ref{intro}.  In later subsections, we will explore the mean exponent for more general
notions of class groups.  

\

We will need the following Lemma of Brauer.

\begin{lemm}\label{brauer}
There is an absolute constant $C_0 > 0$ such that for all number fields $\K$, $\log(h_\K) \leq C_0 \log |\d(\K)|$.
\end{lemm}
\begin{proof}
By Lemma 2 in Chapter 16 of Lang \cite{Lang}), there is an absolute positive constant $C$ such that for all number fields $\K$, 
$\log(h_\K \Rg_\K) \leq C \log |\d(\K)|$.  We can essentially suppress the contribution of the regulator thanks to Friedman's result \cite{Friedman} that for all number fields, we have $\Rg_\K>0.1$.  Thus, by replacing $C$ by a larger constant $C_0$, we have $\log(h_\K) \leq C_0 \log |\d(\K)|$.
\end{proof}

\begin{prop}\label{propmain}
Suppose $\k$ is a number field and $T$ is a finite set of primes  such that $\k_\emptyset^T/\k$ is
infinite.  Suppose $t_0:= |T| -( r_1(\k) + r_2(\k) + 1) > 0$,
and that $\k$ admits a cyclic degree $p$ extension $K$ in which all the primes in $T$ ramify.  
Then the Hilbert $p$-class field tower of $\K$ is infinite with bounded asymptotic mean exponent
$$\Minf(\Gal(\K_\emptyset^\emptyset/\K))< \frac{C_0}{t_0}  \log_p |\d(K)|,$$ where $C_0$ is the 
constant appearing in Lemma \ref{brauer}.
\end{prop}
\begin{proof}
Consider a tower $(\k_n)$ inside $\k_\emptyset^T/\k$ and let $\K_n=\K\k_n$.  To simplify the notation, let $d_n=d(A(\K_n))$ be the $p$-rank of the 
class group of $\K_n$.
By Theorem \ref{genre} applied to $\K_n/\k_n$, we have
\begin{equation}\label{eq1}
d_n \geq |T|[\k_n:\k] - (r_1(\k_n) + r_2(\k_n) + 1)\geq t_0 [\K_n:\K].
\end{equation}
By the definition of the mean exponent $\Mo(\K_n)$, we have $$d_n \Mo(\K_n) = \log_p |A(\K_n)| \leq \log_p h_n$$
where $h_n$ is the class number of $K_n$. Now, if we apply Lemma \ref{brauer}, we have \begin{equation}\label{eq2}
d_n \Mo(\K_n) \leq  \log_p h_n
\leq C_0  \log_p |\d(\K_n)|.
\end{equation}
 But since $\K_n / \K$ is
unramified, $\log_p |\d(\K_n)| = [\K_n:\K] \log_p |\d(K)|$.  Putting the inequalities (\ref{eq1}) and (\ref{eq2}) together, we conclude that
$$
t_0 [\K_n:\K] \Mo(\K_n) \leq C_0 [\K_n:\K] \log_p |\d(\K)|,
$$
hence $\Mo(\K_n)$ is bounded from above by $C_0 \log_p |\d(\K)|/t_0$.  We conclude that
$$
\Minf(\K_\emptyset^\emptyset/\K) \leq \frac{C_0}{t_0} \log_p |\d(\K)|.
$$
\end{proof}
\begin{proof}[Proof of Theorem \ref{intro}.3]
Suppose $\{\ell_1, \ell_2, \ldots, \ell_r\}$ is a large set of primes congruent to $1 \bmod p$.  Let $\k$ be a cyclic degree $p$ extension of $\Q$ in which $\ell_1, \ldots, \ell_{r}$ ramify.  
Consider primes $q_1< q_2$ which split completely in $\k(\zeta_p)/\Q$ if $p$ is odd and in $\k(\zeta_4)/\Q$  if $p=2$.  Let $\k'$ be a cyclic degree $p$ extension of $\Q$ in which $q_1$ and $q_2$ ramify.
Let $T$ be the union of the primes of $\k$ lying over $q_1$ and those lying over $q_2$.  As specified in Theorem \ref{theocritere}, if $r$ is sufficiently large, $\k_\emptyset^T/\k$ is infinite.  Now we let $\K=\k \k'$.  This puts us in the situation of Proposition \ref{propmain}, so gives the desired outcome.
\end{proof}

\subsection{On the mean exponent for $T$-class groups mod $S$ }

In this section, we will expand our notion of class group in two directions: we will look at ($p$-parts of) ray class groups of tame conductor (i.e. a conductor which is a finite product of distinct prime ideals co-prime to $p$), and with the underlying ring being the $T$-integers.  

\begin{defi}
 Let $T$ and $S$ be two disjoints finite sets of places of   $\K$ such that $(S,p)=1$. The mean
 $\Mo_S^T(\K)$ of the invariant factors 
 of the abelian group
  $\A_S^T:=\GST^{ab}$ is defined by $$\displaystyle{\Mo_S^T(\K) := \Mo_{A_S^T}=\frac{a_1+\cdots +a_{d}}{d}}=\log_p |\A_S^T|^{1/d},$$
 where $d=d_p \GST=d_p\A_S^T$ and $ \A_S^T
 \simeq \Z/p^{a_1}\Z \times \cdots \times \Z/p^{a_{d}}\Z$ with:   $1\leq a_1 \leq \cdots \leq a_d$. 
 Note that $\Mo_S^T(\K)=0$ if $|\A_S^T|=1$.\end{defi}

\begin{rema}
Note that $\Mo_S^T$ is well-defined because, thanks to the choice of $S$ being away from $p$, the group $\GST^{ab}$ is finite.
 Clearly,   when $\A_S^T$ is not trivial, we have $\Mo_S^T(\K)\geq 1$.
\end{rema}

\begin{exem}[Iwasawa Theory context]
Let $\LL=\L/\K$ be a $\Z_p$-extension. Let $\K_n$ be the unique
 subfield of $\LL$ of degree $p^n$ over $\K$. Denote by $\X_S^T$ the projective limit of the $p$-group $\A_S^T(\K_n)$ along $\LL$. Then 
$\X_S^T$ is a $\Z_p[[T]]$-module of finite rank and there exist invariants $\mu, \lambda \geq 0$ such that for $n\gg 0$,
$$\log_p |\A_S^T(\K_n)|= \mu p^n + \lambda n + \nu,$$ with 
$\nu \in \Z$. Moreover, $$d_p \A_S^T(\K_n)=s p^n + \lambda + c, $$ where $c \geq 0$ and where $s$ is the $\fq_p[[T]]$-rank of the module
$\fq_p\otimes \X_S^T$.

\begin{prop}\label{zpextension}
Along a  $\Z_p$-extension $\LL$, one has $$\Mo_S^T(\K_n) \sim_{n \rightarrow \infty} 
\left\{\begin{array}{l} \delta \log_p [\K_n:\K]    \ \ {\mbox if \ } \mu=0 \ and \ \lambda \neq 0 \\
           \mu/s                 \ \ {\mbox if \ } \mu \neq 0    \\
           \nu/c \  \ {\mbox if \ } \mu=\lambda= 0
\end{array} \right.$$
 where $\delta= \lambda/(\lambda +c)$ satisifies $0<\delta \leq 1 $.
\end{prop}

\begin{proof}
It is a consequence of the structure theorem of Iwasawa Theory and the fact that $\mu=0$ if and only if $s=0$.
\end{proof}

\begin{rema}
Note when $\mu=0$ and $\lambda\neq 0$, $\Mo_S^T(\K_n)$ is unbounded. This will be in contrast to the examples of section \ref{exemples}.  
\end{rema}

\end{exem}

\medskip

From now on, we want to study the quantity $\Mo(\LL)$  in some tower $\LL$ when the ramification is tame.
First, some definitions.

\begin{defi}
 Let $\LL:=\L/\K$ be an (infinite) extension and let $T$ and $S$ be two finite sets of places of  $\K$ with $(S,p)=1$.
 Put  $$\Msup(\LL,S,T):= \limsup_n \ \Mo_{S,n}^T,$$
 and $$\Minf(\LL,S,T):= \liminf_n  \ \Mo_{S,n}^T,$$
 where $$\Mo_{S,n}^T=\min_{\K_n}\ \Mo_S^T(\K_n),$$ the minimum being taken over all subfields $\K_n$ in $\LL$ of degree $p^n$ over $\K$.
 When $S=T=\emptyset$, we have \ $\Minf(\LL,\emptyset,\emptyset)=\Minf(\LL)$, where $\Minf(\LL)$ was defined in the Introduction.  We also write $\Msup(\LL):=\Msup(\LL,\emptyset,\emptyset)$.
\end{defi}

\begin{rema}
We have  $\limsup_n  \min a_1(\K_n) \leq \Msup(\LL)$ and $\liminf_n \min a_1(\K_n) \leq \Minf(\LL).$
\end{rema}

\begin{defi}
A tower $(\K_n)$ is said exhaustive in  $\LL$ if:
\begin{enumerate}
\item[(i)] $\bigcup \K_n = \LL$,
\item[(ii)] for all $n$, $[\K_{n+1}:\K_n]=p$.
\end{enumerate}
\end{defi}

\begin{prop} \label{propositionsupinf}
For a subtower $(\K_n)$ of $\LL$, $\Minf(\LL,S,T) \leq \liminf_n \Mo_S^T(\K_n)$. If moreover, the subtower $(\K_n)$ is exhaustive in $\LL$ then
 $\Msup(\LL,S,T) \leq \limsup_n \Mo_S^T(\K_n)$.
\end{prop}

\begin{proof} Follows easily from the definitions.
\end{proof}

\subsection{Bounds for mean exponents in tamely ramified towers }


\begin{defi}
For a finite set $S$ of prime ideals of $\K$ satisfying $(S,p)=1$, we put $$\d({\K,S}):= |\d(\K)| \prod_{\p \in S} \N(\p).$$
\end{defi}

A local computation shows the following:

\begin{prop}\label{rootdiscriminantborne} If $S$ is a finite set of prime ideals of $\K$ satisfying
$(S,p)=1$,  the root discriminant remains bounded inside $\K_S^\emptyset/\K$; in other words, 
$\K_S^\emptyset/\K$ is asymptotically good.  Indeed, for a tower $(\K_n)$ in $\K_S^\emptyset/\K$,  we have
 $$\log |\d({\K_n})| \leq [\K_n:\K] \log \d({\K,S}).$$  
\end{prop}

\begin{proof}  See for example Lemma 5 of \cite{HMcompositio}. \end{proof}

\begin{defi}\label{ap}
For a prime $\p$ of $\K$ not dividing  $p$, let  $a(\p):=v_p(\N(\p)-1)$ be the $p$-valuation of $\N(\p)-1$, where $\N(\p)$
 is the absolute norm of $\p$.
\end{defi}

\begin{lemm}\label{cardinal}
 Let $\L/\K$ be a  finite Galois $p$-extension and  let  $S$ be a finite set of places of $\K$ prime to $p$.

 (i) If $p>2$, then  $$|\A_S^T(\L)| \leq |\A (\L)| \left(\prod_{\p\in S} p^{a(\p)}\right)^{[\L:\K]}.$$
 
 (ii) For $p=2$, one has  $$|\A_S^T(\L)| \leq |\A (\L)| \left(\prod_{\p\in S} p^{a^*(\p)}\right)^{[\L:\K]},$$
 where $a^*(\p)=a(\p)$ if $\N(\p) \equiv 1 \bmod 4$ (i.e. if $a(\p)>1$), otherwise $N(\p)=1+2n$, where $n$ is odd and then $a^*(\p)=v_2(1+n)
 + 1$.
 \end{lemm}

 \begin{proof}
 
  One has to give an upper bound  of the tame part of the inertia group of a place  $\P|\p$ in an abelian extension of  $\L$.
  We recall that this inertia group is a quotient of the multiplicative group of the finite field
 $\fq_\P$ of order $N(\P)$. By multiplicativity one can assume that $\L/\K$ is a cyclic degree $p$-extension.
  When $\fq_\P=\fq_\p$, that  means that $\p$ is splitting or is ramified in  $\L/\K$,  then
  $\displaystyle{\prod_{\P|\p} p^{a(\P)} }$ divides $p^{p a(\p)}$ (with equality if $\p$ splits).
  Otherwise, $[\fq_\P:\fq_\p]=p$ and then one note 
  that if $p$ is odd (or when $p=2$ and $N(\p) \equiv 1 \bmod 4$) then $a(\P)=a(\p)+1$. Indeep, if $\fq_\p=\fq_q$, then
  $\fq_\P=\fq_{q^p}$. Let us write $q=1+p^kn$, with $(n,p)=1$.
  Then $\fq_{q^p}^\times$ is cyclic of order 
  $$\begin{array}{rcl}q^p-1&=&(q-1)(q^{p-1}+\cdots + q +1 ) \\
  &=& p^{k+1}n\left(1+np^{k-1}+\cdots + n(p-1)p^{k-1}+p^k A \right)\\
 &=& p^{k+1}n\left(1+\frac{1}{2} n(p-1)p^{k}+p^k A\right)
        \end{array}$$
        where $A \in \Z$, and then $v_p(q^p-1)=p^{k+1}$ for  $p$ odd (and for $p=2$ if $k>1$).

  When $p=2$ with  $N(\p)= 1+2n$,  $n$ odd, one has $a(\P)=v_2(1+n) + 1$.   We leave the remaining details to the reader.
 \end{proof}

 \begin{defi}\label{aS}
  For $p>2$, put $$a(S)=\sum_{\p\in S} a(\p).$$ For $p=2$, put $$a(S)= \sum_{\p\in S} a^*(\p).$$
 \end{defi}

 \begin{rema}
 For $p=2$, one remarks that if the  place $\p$ splits completely in $\L/\K$ then  the ``local factor'' $a^*(\p)$   can be taken $a^*(\p)=a(\p)$.
\end{rema}

\begin{prop} \label{proposition2.15} Let $S$ be a finite set of places of $\K$ with $(S,p)=1$ such that $\K_S^\emptyset/\K$ is infinite.   Let $(\K_n)_n:=\LL$ be a tower in $\K_S^\emptyset/\K$.  Let $T$ and $\Sigma$ be two
other sets of places of $K$; we assume that $(\Sigma,p)=1$ but the cases $\Sigma=\emptyset$ and $S=\Sigma$ are allowed. Recall that $h_n$ denotes the class number  of $\K_n$,  and that $g_n=\log |\d(\K_n)|^{1/2}$ denotes its genus.  Let $d_n=d(\A_\Sigma^T(\K_n))$ be the $p$-rank of $\A_\Sigma^T(\K_n)$.
 Then
 \begin{enumerate}
 \item We have
 $$
 \Mo_{\A_\Sigma^T(\K_n)} \leq 
 \frac{ [\K_n:\K]}{d_n} \left( {\log_p \d(K,S)^{1/2}}  \cdot \frac{\log(h_n)}{g_n}  +  a(\Sigma)
\right).
  $$
\item With $C_0$ denoting the constant from Lemma \ref{brauer}, we have $$
 \Mo_{\A_\Sigma^T(\K_n)} \leq \frac{[\K_n:\K]}{d_n}
\left( C_0 \, {\log_p \d(K,S)} +  a(\Sigma)
\right).
  $$
 If, in addition, there is an $\varepsilon>0$ such that $d_n \geq \varepsilon [\K_n:\K]$ for all $n$, 
 then $ \Mo_{\A_\Sigma^T(\K_n)} $ is bounded as $n\to \infty$.
\end{enumerate}
\end{prop}
\begin{proof} Recall that by Proposition \ref{rootdiscriminantborne}, the genus 
$g_n = \log |\d(\K_n)|^{1/2}$ of $\K_n$ satisfies
\begin{equation}\label{eq3}
g_n \leq [\K_n:\K] \log \d(\K,S)^{1/2}.
\end{equation}
Thanks to Lemma \ref{cardinal}, we have
\begin{eqnarray*}
\log_p |\A_\Sigma^T(\K_n)| &\leq& \log_p |A(\K_n)| + [\K_n:\K] a(\Sigma) \\
&\leq& \log_p h_n +  [\K_n:\K] a(\Sigma)  \\
&\leq& g_n \frac{\log_p(h_n)}{g_n} + [\K_n:\K] a(\Sigma).
\end{eqnarray*}
Now we apply (\ref{eq3}) to the right hand side to find
\begin{eqnarray*}
\log_p |\A_\Sigma^T(\K_n)| &\leq& [\K_n:\K] \left( \frac{\log \d(K,S)^{1/2}}{\log p} \cdot \frac{\log(h_n)}{g_n} +  a(\Sigma) \right). \\
\end{eqnarray*}
It remains only to divide both sides by $d_n$ to obtain the desired inequality.  For the second claim, we merely apply the bound from Lemma \ref{brauer} to the bound from the first claim.
\end{proof}

Before stating the key result of this section, we make a couple of definitions.

\begin{defi} \label{croissancelineaire} In a tower $(\K_n)_n$, and fixing auxiliary finite sets $\Sigma$ and $T$ of places of $K$, 
 one says  that the $p$-rank $d_n$ of $\A_\Sigma^T(\K_n)$ grows $\varepsilon$-linearly  with respect to the degree (for some $\varepsilon >0$) if 
  for  $n\gg 0$ $$d_n \geq \varepsilon [\K_n:\K].$$
\end{defi}

\begin{defi}\label{CC}
Given a real number $A$, a number field $\K$ of signature $(r_1, r_2)$ and a finite set $S$ 
of places of $K$ coprime to $p$, let us define $$\CC(A,\K,S)= A \log \sqrt{\d({\K,S})} - \frac{r_1}{2}(\gamma+1+ \log \pi ) -r_2 (\gamma + \log 2).$$
\end{defi}

\begin{theo}\label{maintheorem-1} We maintain all the hypotheses and notation of Proposition \ref{proposition2.15}.  We
assume that there exists $\varepsilon>0$ such that $d_n \geq \varepsilon [\K_n:\K]$ for all $n$.  If
the conditions of Theorem \ref{TV} apply to $(\K_n)$, then
$$
\limsup_n  \Mo_{\A_\Sigma^T(\K_n)} \leq \frac{1}{\varepsilon} \left( \frac{\alpha(B(\LL), \K, S)}{\log p} + a(\Sigma) \right).
$$
Consequently, 
 $$\Minf(\LL,\Sigma,T) \leq \frac{1}{\varepsilon} \left( \frac{\CC(B(\LL),\K,S)}{\log p} + a(\Sigma) \right).$$
 If moreover the tower $(\K_n)_n$ is exhaustive in $\LL$, then one can replace $\Minf$ by $\Msup$. 
\end{theo}
\begin{proof}
 We begin with the inequality of Proposition \ref{proposition2.15} but introduce the contribution of the regulator, as follows.
 \begin{eqnarray*}
  \Mo_{\A_\Sigma^T(\K_n)}  &\leq&  \frac{[\K_n:\K]}{d_n} \left( \frac{\log \d(K,S)^{1/2}}{\log p} \left( \frac{\log(h_n\Rg_n)}{g_n} - \frac{\log(\Rg_n)}{g_n}\right) +  a(\Sigma)
\right).
\end{eqnarray*}
By hypothesis, we have $[\K_n:\K]/d_n\leq 1/\varepsilon$.  By Theorem \ref{TV}, $\log(h_n\Rg_n)/g_n$
tends to $B(\LL)$.  The last ingredient is a theorem of Zimmert \cite{Zimmert} (we use the enhanced version proved by Tsfasman-Vladut \cite{TV}[Theorem 7.4]):
 $$\liminf_n \log \Rg_n/g_n \geq (\log \sqrt{\pi e}+ \gamma/2)\phi_\Reel + (\log 2 +\gamma)\phi_\Com.$$
 Recalling the definition of $\phi_\Reel, \phi_\Com$, and noting that $r_i(\K_n)=[\K_n:\K] r_i(\K)]$
 for $i=1,2$, we find, after applying Proposition \ref{rootdiscriminantborne}, that
 $$
 \phi_\Reel \geq \frac{r_1(K)}{\log \sqrt{\d(\K,S)}}, \qquad \phi_\Com \geq \frac{r_2(K)}{\log \sqrt{\d(\K,S)}}.
 $$
 Putting all of this together and taking $\limsup_n \Mo_{\A_\Sigma^T(\K_n)}$, we obtain the bound sought.
 \end{proof}

Since it will be the form in which we will apply it most frequently, we will state the following immediate corollary of the theorem.

\begin{coro}\label{corollaire2.16}Suppose in the theorem, we have $S=\Sigma=T=\emptyset$.  Then, assuming the conditions of Theorem \ref{TV} apply to a tower $\LL$ inside $\K_\emptyset^\emptyset/\K$, we have
$$\Minf(\G_\emptyset^\emptyset)
\leq \Minf(\LL,\emptyset,\emptyset) \leq \frac{1}{\varepsilon \log(p)} 
\left( 
\frac{B(\LL)}{2} \log {|\d(\K)|} -\frac{r_1}{2}(\gamma+1+\log \pi) - r_2(\gamma + \log 2)
\right).$$
\end{coro}

\begin{rema}
The comparison of the above Corollary to Proposition \ref{propmain} illustrates how the Tsfasman-Vladut
theorem allows us to give an improved upper bound for the mean exponent.
\end{rema}




\section{Refined Estimates Via Tsfasman-Vladuts}
 \label{exemples}
 
We want to  illustrate the previous section  with a  few examples where we have optimized
the quantity  $B(\L/\K)$ by employing the techniques of  Tsfasman and Vladut \cite{TV}.

\subsection{Tsfasman-Vladuts Machinery}
Let us fix an asymptotically exact extension $\LL:=\L/\K$.
Estimating the constant  $B(\L/\K)$ given by Theorem \ref{TV} is an interesting problem, involving certain kinds of optimization.
Indeed the quantity for which we would like to have a tight upper bound is the sum
$$\sum_q b_q \phi_q -b_0\phi_\Reel -b_1\phi_\Com$$
satisfying the three following conditions:
\begin{enumerate}
 \item[(i)]$\phi_q >0$ ;
 \item[(ii)] $\forall \ell, \  \sum_m m\phi_{\ell^m} \leq \phi_\Reel + 2 \phi_\Com$,
 \item[(iii)]$ \sum_q a_q \phi_q +a_0 \phi_\Reel 
+a_1 \phi_\Com \leq 1,$
\end{enumerate}
where $$b_q=\log\frac{q}{q-1} , \ \ \ a_q=\frac{\log q}{\sqrt{q}-1},$$
$$a_0= \log2\sqrt{2\pi}+\pi/4+\gamma/2, \ \ \ a_1= \log(8\pi) +\gamma,$$
$$b_0=\log 2, \ \ \ b_1=\log 2\pi.$$

One now replaces  each $\phi_q$ by a variables $x_q$ and the problem becomes a question of linear optimization.  For convenience, we put $x_0=\phi_\Reel$ and $x_1=\phi_\Com$.

\medskip

One studies the quantity \ $\sum_q b_q x_q -b_0x_0 -b_1x_1$ \ when  $x_0$ and $x_1$ are fixed ({\it i.e.} 
when for example one has a totally real tower or a totally complex tower). Similarly,  one can exploit knowledge of any finite place that is totally split in $\LL$.
One can also use some information coming from the base field $\K$:  typically if the base field
has no place of norm  $\ell$,  then $x_\ell$
would be fixed and equals to~$0$.

\medskip

Denote by $\Sigma=\{ q_1,\cdots, q_r\}$ a set of powers of prime numbers for which one fixes  $x_{q_i}$.
We want to give   an upper bound as small as possible of the quantity 
$$\sum_{q \notin \Sigma} b_q x_q,$$
with the conditions
$$(i)' \ x_q >0, \ \ \  (ii)' \ \sum_m m x_{q^m} \leq x_0+2x_1, \ \ \ (iii)' \ \sum_{q\notin \Sigma} a_q x_q \leq 1- \sum_{q \in \Sigma} a_q x_q.$$

As explained in \cite{TV}, there are two reductions: first, one can assume that  $x_{\ell^*}$ attains  the maximum for condition
$(ii)'$, where  $\ell^*$ is the smallest power of $\ell$ for which  $x_{\ell^*}\neq 0$; 
try to optimize inequality $(iii)'$ for the smallest powers  $\ell^*$.

Now let $\ell_0^{*}$ the smallest power   such that
$$\sum_{\ell^* < \ell^{*}_0} (x_0+2x_1-\varepsilon_{\ell^*})a_{\ell^*} \leq 1 - (a_0x_0+a_1x_1+ \sum_{q\in  \Sigma} a_{q} x_q),$$
where   $\varepsilon_{\ell^*} \leq x_0 +2x_1$ is a constraint of  $\ell$ related to the base field.

Let  $\alpha \in [0,1)$ such that 
$$\alpha(x_0+2x_1-\varepsilon_{\ell_0^*})a_{\ell_0^*}=1-a_0x_0-a_1x_1 -\sum_{q\in  \Sigma} a_{q} x_q-\sum_{\ell^* < \ell^{*}_0} (x_0+2x_1-\varepsilon_{\ell^*})a_{\ell^*}.$$

\begin{prop} One has:
 $$\sum_{q} b_q \phi_q \ \leq \ \sum_{q\in \Sigma}b_q x_q + \sum_{\ell^*< \ell^*_0} (x_0+2x_1-\varepsilon_{\ell^*})b_{\ell^*}+
 \alpha(x_0+2x_1-\varepsilon_{\ell_0^*})b_{\ell_0^*}.$$
\end{prop}

\medskip

\subsection{Strategy for Construction of Examples}

Below we will study some examples built with the following strategy. First $p=2$. Let $\k/\Q$ be a real quadratic field. Suppose that 
for the set $T$ of places of $\Q$, the $2$-tower $\k_\emptyset^T/\k$ is infinite (for doing this, we apply Corollary \ref{critere}).
Consider then $\K:=\k(\sqrt{-D})$, where $D=\prod_{\p\in T}\p$; put  $\LL:=\K\k_\emptyset^T$. Take an exhaustive tower $(\k_n)_n$ of $\k_\emptyset^T/\k$, then
$\K_n:=\k_n\K$ is an exhaustive tower of $\LL$. Moreover, $(\K_n)$ is a subtower of $\K_\emptyset^T$.
And then, by corollary \ref{corollaire2.16} one obtains  bounds for $ \Msup(\LL)$ and $\Minf(\K_\emptyset^\emptyset/\K)$.

\medskip

\subsection{Examples} In all of the examples below, we fix $p=2$, since in this case, we can employ ramification at infinity in conjunction with the genus theory bounds.

\begin{exem}
Let  $\k=\Q(\sqrt{8 \cdot 5 \cdot 7 \cdot 11 \cdot 13 \cdot 17 \cdot 19 \cdot 23})$.
Thanks to Corollary \ref{critere},  the number field $\k$ has an infinite $2$-extension $\k^T/\k$  ($S=\emptyset$)  
where $T=\{\ell_9\}$ 
is the set containing the only  place  above
 $3$ (of norm $9$).
Put $\K=\k(\sqrt{-3})$.
Denote by $(\k_n)_n$ a  tower of $\k^T$; put $\K_n=\K\k_n$ and $\displaystyle{\L=\bigcup_n \K_n}$ and $\L/\K:= \LL$. 
  Then by Genus
Theory (cf Theorem 
\ref{genre}) along $\k^T/\k$, one obtains that $$d_n = d_2 \A(\K_n) \geq  [\K_n:\K] -1.$$

If we apply Corollary \ref{corollaire2.16}, we find  $$\Minf(\K_\emptyset^\emptyset/\K) \leq \Msup(\LL) \leq \frac{1}{22 \cdot \log 2}\left( B \log \sqrt{|\d(\K)|} -(\gamma + \log 2)\right)
\approx  30.683 \cdots $$
where here one has taken $B\approx 1.0938$.
But we can do better by applying the refined results of Tsfasman-Vladut. The base field  $\K$ is of degree  $4$ over $\Q$.
 The tower we consider is totally complex  and by construction the prime  $\ell^*=9$ (over $3$ with norm $9$)  
 splits completely in the considered tower.
 Here $g_\K = \log(\sqrt{8\cdot 3 \cdot 5 \cdot 7 \cdot 11 \cdot 13 \cdot 17 \cdot 19 \cdot 23})$.
 In order of increasing size of the norm, one has ideals of norm: $4$, $7$, $7$,
 $9$,  $13, 13,19, 19$, $25$, $31, 37,43, 43, 43, 43$ etc.
 
 \
 
 One fixes the following  conditions $x_0=0$, $x_1=r_2/g=2/g$, $x_2=0$, $x_3=0$, $x_5=0$, $x_9=1/g=x_1/2$.
 One considers  $\Sigma=\{9\}$. Moreover  $x_4 \leq 1/g = x_1/2$,  $\varepsilon_{2^*}=x_1$ and $x_{25} \leq 1/g$.
 One has $$ \begin{array}{c}\displaystyle{ g - 2(\gamma + \log(8\pi)) -\a{9} -2 \left( \a{7}+\a{13}+\a{19}\right)}\\ 
 \displaystyle{ - \left(\a{4}+\a{25}\right)-4\a{31} \ < \ \a{37},}\end{array}$$
 and then $\ell_0^*=37$.
 One obtains
 $$\begin{array}{c}\displaystyle{ B(\L/\K) \leq 1 - \frac{r_2}{g}\log 2\pi +  \frac{1}{g} \left( \log(4/3)+  \log(9/8) + \log(25/24) \right. }\\ \displaystyle{\left.
 + 2\log(7/6)+2\log(13/12)+2\log(19/18)+4\log(31/30) + 4 \alpha \log(37/36)\right),}
 \end{array}$$
 where $$\begin{array}{rl} \displaystyle{ 4 \alpha \a{43} = }&\displaystyle{g-2(\log 8\pi +\gamma) -\frac{\log 9}{2}-\log 4 -\frac{\log 25}{\sqrt{25}-1}}
 \\ &\displaystyle{-2 \left(\frac{\log 7}{\sqrt{7}-1}+
 \frac{\log 13}{\sqrt{13}-1} + 
 \frac{\log 19}{\sqrt{19}-1}  +2 \frac{\log 31}{\sqrt{31}-1}\right)}.
 \end{array}$$
 and then $B(\L/\K)\approx 0.878 \cdots $, and 
 $$\Minf(\K_\emptyset^\emptyset/\K)  \leq \Msup(\LL) \leq 24.100.$$
\end{exem}

\begin{exem}
Let $\k$ be the real quadratic field of discriminant $D$ where $D$ is the the product of the elements in the set $$U=\{47, 59,  61 , 67 , 71 ,  73 ,  79 ,  83 ,  89 ,  97 ,  101 ,  103 ,  107
 ,  109 ,  113 ,  127 ,  131,  137 ,  139 ,  149,  151\}.$$
Let $T_{\mathrm{in}}=\{3,7,29,31,37,41,43,53\}$, $T_{\mathrm{dec}}=\{2,5,11,13,17,19,23\} $; put  $T=T_{\mathrm{in}}\cup T_{\mathrm{dec}}$; $|T|=22$.
The places of  $T_{\mathrm{in}}$ are inert in $\k/\Q$ and the places of  $T_{\mathrm{dec}}$ are totally decomposed in $\k/\Q$.
One uses Corollay \ref{critere}: the number field  $\k$ has an infinite  $T$-split $2$-tower $\k^T/\k$.
Consider now the number field $\K=\k(\sqrt{-D})$, where $D= \prod_{\ell \in T} \ell$ and put $\L=\K\k^T$.
Then for all number fields $\K_n$ along  $\L/\K$, one has  $$d_2 \Cl_{\K_n} \geq 22 [\K_n:\K]-1.$$
Then  $$\Mo(\K_n)\leq \frac{1}{22 \log 2} \cdot \left( B  \log \sqrt{|d_\K|} - (\gamma + \log 2)\right)  \approx 9.662 \cdots $$

We now use the stategy of Tsfasman and  Vladut to optimize $B(\L/\K)$.
Each place of $T$ splits totally in $\L/\K$: the associated parameters $\phi_{\ell^*}$ are then fixed.
More precisely, for every  $\ell \in T_{\mathrm{in}}$, we have $\phi_{\ell}=0$, $\phi_{\ell^2}=1/g$ and $\phi_{\ell^i}=0$ for $i>2$; for 
$\ell \in T_{\mathrm{dec}}$, one fixes 
$\phi_{\ell}=2/g$ and $\phi_{\ell^i}=0$ for $i>1$.
Moreover for $\ell \leq 150$, $\phi_{\ell^*} \leq 2/g$. In fact one may be more precise: only the primes of  $R=\{47,49,61,103,113,127,131,139\}$
split (and ramify) the others are inert (with $67^2$ the smallest norm).
One remarks that the sum 
$$A= g-2(\gamma +\log 8\pi) - 2\sum_{\ell \in T_{\mathrm{dec}}}\frac{\log \ell}{\sqrt{\ell}-1} - 
\sum_{\ell \in T_{\mathrm{in}}} \frac{\log \ell^2}{\ell-1} -2\sum_{\ell \in R}\frac{\log \ell}{\sqrt{\ell}-1} \approx 103.774 $$
is smaller than 
 $\displaystyle{ 4 \sum_{\ell \leq 67^2}^* \frac{\log \ell}{\sqrt{\ell}-1} }$
where the last sum is taken over the splitting places  in $\K/\Q$ ({\it i.e.}  $127$ such places).
One finds $\ell_0^*=3877$ and to finish $$A
-4 \  \sum_{153 \leq \ell < 3877}^* \frac{\log \ell}{\sqrt{\ell}-1} 
\approx 0.528 .$$
Here $\alpha \approx 0.980$

After making the  computation of the default, one obtains $$\sum_q b_q\phi_q \leq  3.348$$
and then $B(\L/\K) \leq 1.01421 \cdots $ and  $$\Minf(\K_\emptyset^\emptyset/\K) \leq \Msup(\LL) \leq \frac{1}{\log 2} 6.306 \cdots \approx 9.098\cdots $$
\end{exem}

\begin{exem}
 Let $$\k=\Q(\sqrt{8\cdot 3 \cdot 5 \cdot 7 \cdot 11 \cdot 13 \cdot 17 \cdot 19 \cdot 23 \cdot 29 \cdot  31 \cdot 37 \cdot
 41 \cdot 43\cdot 47 \cdot 53 }).$$ 
 Let $T_{\mathrm{in}}=\{ 71,79,83,97,101\}$ et $T_{\mathrm{dec}}=\{59,61,67,73\}$; $T=T_{\mathrm{in}}\cup T_{\mathrm{dec}}$; $|T|=13$.
 Put  $\K=\k(\sqrt{-59\cdot 61 \cdot 67 \cdot 71\cdot 73 \cdot 79 \cdot 83\cdot 97 \cdot 101})$.
 The number field  $\k$ has an infinite  $2$-tower $\k^T$; put $\L=\K\k^T$. Along the extension $\L/\K$, one has 
 $$d_2 \A(\K_n) \geq 13 [\K_n:\K]-1.$$
 By looking at the primes $\ell \leq 100$, one see that $$x_2=x_3=x_{7}=x_{19}=x_{29}=x_{31}=x_{41}=x_{47}=x_{53} =0$$
 Here $\ell_0^*=1249$ and so there are $47$  primes that are splitting in $\K/\Q$ and with norm
  less than  $\ell_0^*$.
 One find  $\alpha \approx 1.020$,
 $$\sum_q b_q\phi_q  \leq 2.192 \cdots $$
 and $B(\L/\K) \leq 0.951 \cdots$
To conclude,  $$\Minf(\K_\emptyset^\emptyset/\K) \leq \Msup(\LL) \leq \frac{1}{\log 2} 6.139 \cdots \approx 8.857 \cdots  $$
\end{exem}

\begin{exem} Take $p=2$.
 Let $\k=\Q(\sqrt{2\cdot 3 \cdot 5 \cdot 7 \cdot 11 \cdot 13 \cdot 17  \cdot 19 \cdot 23 \cdot 29 \cdot 31 \cdot 41 \cdot 43 })$.
Put $T_{\mathrm{dec}}=\{59,61\}$ and $T_{\mathrm{in}}=\{37,47,53,67,89\}$; $|T|=9$.
Let us consider $\K=\k(\sqrt{-37 \cdot 47 \cdot 53 \cdot 59 \cdot 61 \cdot 67 \cdot 89})$. Along the extension $\L/\K$, one has 
 $$d_2 \A(\K_n) \geq 9 [\K_n:\K]-1.$$
Here $$x_2=x_3=x_7=x_{13}=x_{31}=x_{37}=x_{47}=0,$$  $\ell_0^*=647$ and $\alpha \approx 0.072$.
Then  $\sum_q b_q \phi_q \leq 1.993 \cdots$, $B(\L/\K) \leq 0.9733 \cdots $ and 
$\Minf(\K_\emptyset^\emptyset/\K) \leq \Msup(\LL) \leq 9.657 \cdots$.
 \end{exem}

\begin{exem}
 
Take $p=2$. Let
$$\k= \Q(\sqrt{8\cdot 3 \cdot 5 \cdot 7\cdot11\cdot13\cdot17\cdot19\cdot23\cdot29\cdot31\cdot37\cdot41\cdot43\cdot47\cdot53\cdot59\cdot61\cdot67\cdot71\cdot73})$$
Put $T_{\mathrm{dec}}=\{79,83,89,97,107,109,137\}$, $T_{\mathrm{in}}=\{101,103,113,127,131,149,157,173\}$.
Let $D$ be the product of the elements in $T_{\mathrm{dec}}$ and $T_{\mathrm{in}}$ and let
$\K=\k(\sqrt{D})$.
Here $d_2 \A(\K_n) \geq 20[\K_n:\K]-1$. Finally, for this example, 
 $\ell_0^*=1069$, $B(\LL) \leq 1.013\cdots $ and thus $$\Minf(\K_\emptyset^\emptyset/\K) \leq \Msup(\LL) \leq 10.022 \cdots$$
\end{exem}

\section{Linear growth of the $p$-class rank}

\subsection{The mean $\Mo$ and a question of Ihara}
The examples of the previous section show how primes that split completely can be used to produce towers with linear growth for
the $p$-rank of the class group, which then places constraints on
the asymptotic mean  $\Minf$.
In particular, with the help of Proposition \ref{proposition1.2}, we have the following result.

\begin{prop} \label{propositionsplit}
 Let  $S$ and $T$ be two sets of places  of $\K$, $(S,p)=1$.
 For all subfields  $\K_n$ of $\K_S^T$, one has  $$d_p \A_T(\K_n) \geq [\K_n:\K]\Big(|T|-(r_1(\K)+r_2(\K))\Big).$$
Note that by the Golod-Shafarevich criterion (see Theorem \ref{GS} and Proposition \ref{boundrelations}), $K_S^T/K$ is infinite once  $|S|$ is large as compared to  $|T|$, 
 and in this case
 $$\Minf(\K_S^T/\K,T,\emptyset) \leq \frac{1}{|T|-(r_1(\K)+r_2(\K))} \left(\frac{\CC(B(\K_S^T/\K),\K,S)}{\log p} + a(T) \right),$$
 where $a(T)$ is given in Definition \ref{ap} and \ref{aS}.
\end{prop}

\begin{proof}
 It is an application of Theorem \ref{maintheorem-1} with $\varepsilon=|T|-(r_1(\K)+r_2(\K))$.
\end{proof}

\medskip

At this point, let us recall a question of Ihara \cite{Ihara}:
\begin{Question}
What can one say about the number of primes that decompose completely in an infinite unramified Galois extension?
\end{Question}

The importance of the above question for the invariant $\Mo$ is illustrated in the following Corollary.

 \begin{coro}
  Suppose that  in the  pro-$p$-extension  $\K_S/\K$, with $(S,p)=1$,  the set $\T$ of places that split completely 
  in this tower is infinite. Then for all $\varepsilon>0$, by taking  large $T \subset \T$ , one obtains
  $$1\leq \Minf(\K_S/\K,T,\emptyset) \leq \frac{a(T)}{|T|} + \varepsilon.$$ 
  If moreover the set $\T$ contains infinitely many primes $\p$ with $a(\p)=1$ then, by choosing $T$ to consist only of such primes, we can arrange $\Minf(\K_S/\K,T,\emptyset)$  to be as close to $1$ as desired.
 \end{coro}

\subsection{Ershov's trick} 
 
Thanks to a result of Schmidt \cite{schmidt}, the phenomenon of  Proposition \ref{propositionsplit} which we derived from number theory considerations, can be obtained via a clever idea due to Ershov \cite{Ershov} using
pro-$p$-group presentations.

\medskip

Let $\K$ be a number field and $S_0$ a finite set of places of $\K$, $(S_0,p)=1$. We assume that $\delta_\K=0$ and that $\A_\K$ is trivial.
By    \cite{schmidt}, one can choose a finite set $\Sigma$  of places of $\K$ such that 

\begin{enumerate}
\item[(i)] $(\Sigma,p)=1$,  $S_0\subset \Sigma$;
 \item[(ii)] The natural map $\displaystyle{H^2(\G_\Sigma,\fq_p) \stackrel{\sim}{\longrightarrow} \bigoplus_{v \in \Sigma} H^2(\G_v,\fq_p)}$ is an isomorphism;
 \item[(iii)] the pro-$p$-group $\G_\Sigma$ is of cohomological dimension $2$ and 
 $$\chi(\G_\Sigma):= 1 -d_p H^1(\G_\Sigma,\fq_p)+d_p H^2(\G_\Sigma,\fq_p)=r_1(\K)+r_2(\K).$$
\end{enumerate}

\medskip

Put $d=d_p \G_\Sigma$ and  $k=|\Sigma|$. As $\A_\K$ is trivial, $d \leq k$.

 By  (ii)  the relations of $\G_\Sigma$ are all local. 
In fact, by following the proof of Theorem 6.1 of \cite{schmidt}, one can show that there exists a 
subset $S \subseteq \Sigma$ containing $S_0$ with the following property.  Letting $T=\Sigma - S$ and $t=|T|$, there exists a  basis of generators   $(x_i)$
of $\G_\Sigma$ such that
 for $i=1,\cdots, t$, every element $x_i$ is a generator of the inertia group in $\K_\Sigma/\K$ of one place of $T$.
  (The set $S$ allows us to kill a certain Shafarevich group.) The quantities $t$ and $d$ can be  as large as we want.

 \medskip
 
 Hence the group $\G_\Sigma$ can be described by generators and relations as
 $$\langle x_1,\cdots, x_d \ | \ [x_1,F_1]=x_1^{p\lambda_1}, \cdots, [x_t,F_t]=x_t^{p\lambda_t},\ r_{t+1}, \cdots, r_k \rangle,$$
 where the elements
 $F_i$ are lifts of the  Frobenius of the places $v_i \in S$, and $\lambda_i$ belongs to $\Z_p$
 (for $p=2$, $\lambda_i\in 2\Z_2$) and  where we recall that $k= d_p H^2(\G_\Sigma,\fq_p)=|\Sigma|$.
 Note that the relations $[x_i,F_i]x_i^{p\lambda_i},  i=1,\cdots, t$ are the local conditions. 
 
 \medskip
 
Then take a minimal presentation of $\G:=\G_\Sigma$ as follows:
 $$1\longrightarrow R \longrightarrow \F \longrightarrow \G \longrightarrow 1,$$
 where $R$ is the normal subgroup of $\F$ generated by the  relations
 $$\langle [x_1,F_1]=x_1^{p\lambda_1}, \cdots, [x_t,F_t]=x_t^{p\lambda_t}, \ r_{t+1}, \cdots, r_k \rangle.$$
 
 Let  $\H$ be the normal subgroup  of $\F$ generated by the  elements $x_1,\cdots, x_{t}, F_1,\cdots, F_t$. 
 By maximality, the subgroup $\H R$ corresponds to $\G_S^{T}$. 
 Put $\Gamma=\G_S^{T}$.
 
 \medskip
 
 Let now $\Gamma_i$ be an open subgroup of $\Gamma$ and let  $\F_i$ be the normal subgroup of $\F$ containing $\H R$ and satisfying  
 $\F/\F_i \simeq  \Gamma/\Gamma_i \simeq \G/\G_i$, 
 where $\G_i$ corresponds to
 $\F_i / R$.
 Now by  Schreier's formula one has 
 $$d_p \F_i -1=[\F:\F_i](d_p \G -1),$$
by recalling that  $d_p \G = d_p \F$.
 One then has the exact sequence 
 $$1\longrightarrow \F_i^p[\F_i,\F_i] R /\F_i^p[\F_i,\F_i] \longrightarrow \F_i/\F_i^p[\F_i,\F_i] \longrightarrow 
 \F_i/\F_i^p[\F_i,\F_i]R \longrightarrow 1,$$
 where  $\F_i/R \simeq \G_i$.
 Now, by construction, as $\F_i$ contains $\H$, the first generators of $R$ are in $\F_i^p[\F_i,\F_i]$.
 One see  very quicky that the quotient $\F_i^p[\F_i,\F_i] R /\F_i^p[\F_i,\F_i]$ is topologically generated by the elements of the form 
 $yzy^{-1}$, where $y $ is a representative of a class of $\F/\F_i$ and $ z \in \{ r_{t+1},\cdots, r_k\}$: indeed $R \subset \F_i$.
 Thus $$d_p(\G_i) \geq [\G:\G_i] (d - 1- k+t)+1,$$
 and as $1-d+k=\chi(\G_\Sigma)=r_1(\K)+r_2(\K) $,
 one obtains
 $$\frac{d_p(\G_i)}{[\G:\G_i]} \geq t-\big(r_1(\K)+r_2(\K)\big).$$
 Here $\G_i =\G_\Sigma(\K_i)$ where $\K_i$ is the fixed field of $\G_i$ inside the tower  $\K_\Sigma/\K$.

\subsection{On Schreier's  bound}
Recall again the principle behind the construction of the examples of the section \ref{exemples}. Take $p=2$.
Let $\k$ be a real quadratic field having an infinite $2$-extension $\k^T/\k$.
Put $t=|T|- (r_1+r_2)$. Let $\K/\k$ be an imaginary quadratic extension in which all places of $T$ are ramified.
Let $(\k_n)$ be an exhaustive tower in 
$\k^T/\k$ and consider the tower $(\K\k_n)$ of $\K$, which is evidently inside $\K^T/\K$.  By 
Genus Theory applied to each quadratic extension $\K_n/\k_n$, 
$ d_p \Cl_{\K_n} \geq  [\K_n:\K]t -1 $.
In \cite{Hajir}, it has been proven that in fact $$d_p \Cl_{\K_n} \geq  [\K_n:\K]t +1 .$$ 

At this level, we recall that Genus Theory allows us  a lower bound  of the $p$-rank of a subgroup of  $\Cl_{\K_n}$ without taking into account the contribution of 
 $\Cl_{\k_n}$ i.e. $$d_p \Cl_{\K_n} \geq [\K_n:\K]t -1 + \alpha_n,$$ with $\alpha_n \leq d_p \Cl_{\k_n}$ measuring the added contribution to the rank coming from the injection of  $\Cl_{\k_n}$ into
 $\Cl_{\K_n}$ (see \cite{Maire-Bx}).

\medskip

In the other direction, thanks to Schreier's inequality, one has 
$$d_p \Cl_{\K_n} \leq (d_p \Cl_\K-1)[\K_n:\K]+1,$$
and then $$t\  [\K_n:\K] \leq d_p\Cl_{\K_n} -1  \leq (d_p \Cl_\K-1)[\K_n:\K],$$
which naturally raises the following question raised in \cite{Hajir}.

\begin{Question}
 Is it possible to create an example as above having an optimal inequality, {\it i.e.} such that  $d_p \Cl_\K - 1= t$~?
\end{Question}

\medskip

In \cite{Hajir}, it was shown that a sequence of examples with the ratio $(d_p \Cl_\K -1)/t$ tending to $1$ can be created.  In the remainder of this section,  we will make an attempt to find examples with small $(d_p \Cl_\K - 1) - t$  by considering some ray class groups.  

\

We take $p=2$.
To recall a Theorem due to Gras-Munnier (see \cite{gras}, section I.4 or chapter VI  or \cite{Gras-Munnier}), we fix the notation.
Let   $\F':=\F(\sqrt{E},\sqrt{A})$ be the governing field of a number field $F$ where $E$ is the group of units of $\F$, 
where $A=\{a_1,\cdots, a_d\}$, ${\mathcal A_i}^2=a_i \O_\F$, $({\mathcal A_i})_i$ being a system of generator of $\A_\F[2]$.

\begin{theo}[Gras-Munnier] \label{GM}
Let $T=\{\p_1,\cdots, \p_t\}$ be a set of places of $\K$, with $\N \p_i \equiv 1 \bmod p$. There exists an extension $\L/\F$
cyclic of degree $2$, exactly and totally ramified at $T$  if and only if, for $i=1,\cdots, t$, there exists $a_i \in \fq_p^\times$, such that
$$\prod_{i=1}^t \FF{\F'/\F}{\P_i}^{a_i} =1 \ \in \Gal(\F'/\F),$$
where $\P_i$ is an ideal  of $L$ above $\p_i$.
\end{theo}

\medskip

Now, take $\ell$ to be a prime with $\ell \equiv 1 \bmod 32$.
Let $\F$ be the totally real subfield of $\Q(\zeta_\ell)$ of degree $16$ over $\Q$. Let $\{-1,\varepsilon_1,\cdots,\varepsilon_r\}$ be a basis of $E/E^2$. Note that the extension $\F'/\Q$ is a Galois extension and contains $\F$ (here $\F'$ is the governing field defined above). 
By the Chebotarev Density Theorem, we can find an odd prime $q$ that splits completely in $\F'/\Q$. Now by Theorem \ref{GM}, for all  primes $\q_i$  of $\F$ above $q$,
    there exists a cylic $2$-extension
   exactly $\{\q_i\}$-ramified.  We conclude that the $2$-rank of  the $2$-class 
group $\A_S(\K)$ is at least $16$, where $S$ is the set of places of $\K$ above $q$. 
 Moreover by the condition above $q$, one has that $-1$ is  a square in $\Q_q$, that means that $q\equiv 1 \bmod 4$.
 Now, again by applying Chebotarev Density Theorem take $p_1$ that splits completely in the  extension $\F_S^{ab}(\sqrt{-1})/\Q$ as well as another prime $p_2$  
that splits completely in $\F_S/\Q$ but which is  inert in $\Q(\sqrt{-1})/\Q$. 

\

Let  $T$ be the set of places of $\F$ above $\{p_1,p_2\}$. 
Then the $2$-rank of $\G_S:=\Gal(\F_S/\F)$ and the $2$-rank of $\G_S^T:=\Gal(\F_S^T/\F)$ are the same and are at least $16$. Now, $r(\GST) \leq 48$ (see Proposition \ref{boundrelations}) and, by the Theorem of Golod-Shafarevich (see Theorem \ref{TV})  the tower $\F_S^T/\F$ is infinite and then  the tower
 $\Q_\Sigma^T/\Q$ is  infinite too, where $\Sigma=\{q,\ell\}$.
 
 \
 
 Put $\K=\Q(\sqrt{-p_1p_2})$.  The primes $\ell$ and $q$ are split in $\K/\Q$.
As $p_2 \equiv 3 \bmod 4$, one has d$_2 \A_K = 1$  and the
  $2$-rank of the ray class group of $\K$ with modulus 
$q \ell$ is at most $5$.  Now consider the compositum $\L:=\Q_\Sigma^T \K$.  Thanks to Schreier's inequality and to Genus Theory, one has for all number fields $\K_n$ in $\L/\K$:
$$2[\K_n:\K] \leq d_2 \A_\Sigma(\K_n) -1 \leq 4[\K_n:\K].$$
   
   \
   
   By assuming a hypothesis, we can improve the above estimate. Indeed, the $2$-group  $\G:=\Gal(\F/\Q)$ acts on the elementary abelian $2$-group $\H:=\Gal(\F'/\F)$. Hence there exists a subgroup $\H_0$ of $\H$  of order $2$ on which $\G$ acts trivially. 
   
   \
   
   \emph{For the remainder of this section,} suppose that $\H_0$ can be chosen such that 
   $\H_0 \nsubseteq \Gal(\F'/\F(\sqrt{-1}))$.
   
   \

   By the Chebotarev Density Theorem, take an odd prime $q$ such that its Frobenius in $\Gal(\F'/\Q)$ is a generator of $\H_0$.
   \begin{lemm} 
Let $\q_i \neq \q_j$ be two primes of $\F$ above $q$. Then    $\displaystyle{\left(\frac{\F'/\F}{\q_i}\right)=\left(\frac{\F'/\F}{\q_j}\right)}$.
   \end{lemm} 
   \begin{proof}
   The primes $\q_i$ and $\q_j$ are conjugated: there exists $g\in \G$ such that $\q_j=\q_i^g$. We are done thanks to  the property of the Artin Symbol: $\displaystyle{\left(\frac{\F'/\F}{\q_i^g}\right)=g\cdot \left(\frac{\F'/\F}{\q_i}\right) \cdot g^{-1}}$ and the fact that $\G$ acts trivially on $\H_0$.
   \end{proof}
  
   Now by Theorem \ref{GM}, for all pairs of primes $\q_i\neq \q_j$  of $\F$ above $q$,
    there exists a cylic $2$-extension
   exactly $\{\q_i,\q_j\}$-ramified. Then, this implies that the $2$-rank of  the $2$-class 
group $\A_S(\K)$ is at least $15$, where $S$ is the set of places of $\K$ above $q$. 
Moreover by the condition above $q$, one has that $-1$ is not a square in $\Q_q$, that means that $q\equiv 3 \bmod 4$.
We now put $\K=\Q(\sqrt{-p_1p_2})$ and proceed exactly as before;  the
  $2$-rank of the ray class group of $\K$ with modulus 
$q \ell$ is at most $4$ if $q$ is inert in $\K/\Q$ or $5$ if $q$ splits. 

\begin{lemm}
Here, $d_2 \A_{\K,\Sigma} \leq 4$. 
\end{lemm}

\begin{proof} One has only to look at the case  where  $q$ splits in $\K/\Q$. Let $\alpha \in \K$ such that $\alpha$ is the square of the unique non trivial   class $C$ of $\A_\K$: $C^2=(\alpha)$.
Consider the morphism $$\theta : \langle-1,\alpha\rangle \mapsto \frac{\fq_{\l_1}^\times}{\fq_{\l_1}^{\times2}} \times \frac{\fq_{\l_2}^\times}{\fq_{\l_2}^{\times2}} \times \frac{\fq_{\q_1}^\times}{\fq_{\q_1}^{\times2}} \times \frac{\fq_{\q_2}^\times}{\fq_{\q_2}^{\times2}}, $$  where $\l_i$ and $\q_i$ are the primes of $\K$ above $q\ell$ and where $\fq_{\q_i}$ (resp. $\fq_{\l_i}$) is the residue field of $\q_i$ (resp. of $\l_i$).  
Then  one has the formula (see \cite{Maire-PMB} or see \cite{gras}):  $d_2 \A_{\K,\Sigma} = d_2 \A_\K + |\Sigma| - d_2 \Im(\theta)$. Now as $q\equiv -1 \bmod 4$, the image of $\theta$ is at least of order $2$ and then we have down.
\end{proof}

Now consider the compositum $\L:=\Q_\Sigma^T \K$.  Thanks to Schreier's inequality and to Genus Theory, one has for all number fields $\K_n$ in $\L/\K$:
$$2[\K_n:\K] \leq d_2 \A_\Sigma(\K_n) -1 \leq 3[\K_n:\K].$$

\section{Invariant factors in   pro-$p$-groups}

For this section the main reference is \cite{DSMN}.

\

We begin with a straightforward observation.

\begin{prop}
 Let $\G$ be a torsion-free \emph{FAb} pro-$p$-group. Let $(\U)$ be a basis of open  subgroups of $\G$. Then
  the sequence of the exponents  $e(\U^{ab})$ of $\U^{ab}$ is not bounded.
\end{prop}

\begin{proof} 
Suppose  that there exists an integer $k$ such that for all open subgroups $\U$, $e(\U^{ab})\leq k$. Take $1\neq x \in \G$. Then
$\langle x^k\rangle \U \subset [\U,\U]$, that means 
 $$\langle x^k\rangle=\bigcap_\U \langle x^k \rangle \U \subset \bigcap_\U [\U,\U]=\{1\}.$$
 In other words, $x^k=1$ and,  as $\G$ is torsion-free, $x=1$. Contradiction.
\end{proof}

Our work in the previous sections on exponents of $p$-class groups leads us now to defining the following invariant for finitely generated \emph{FAb} pro-$p$ groups.

\begin{defi}
 Let  $\G$ be a \emph{FAb} pro-$p$-group of finite type. 
 For any open subgroup $\U$ of $\G$, since $\U^{\mathrm{ab}}$ is finite, $\Mo_{\U^{\mathrm{ab}}}$
 is well-defined.  For $n\geq 1$, we put
  $$\Mo_n(\G):= \min_{[\G:\U]=p^n}\Mo_{\U^{\mathrm{ab}}}$$
  and then define the asymptotic mean exponent of $\G$ to be $$\Minf(\G):= \liminf_n \Mo_n(\G).$$
\end{defi}

In the remainder of this section, we will show how to estimate the asymptotic mean exponent in  two special cases.

\subsection{In analytic pro-$p$-groups}

As noted by  G\"artner in \cite{Gartner-1}, the exponents of open subgroups of an infinite $p$-adic analytic pro-$p$-group tend to infinity. 
To be more precise, 
let  $\G$ be an analytic pro-$p$-group of dimension $d$. Then $\G$ has  an open  uniform subgroup  $\U$ (of rank  $d$). Put
$\U_1=\U$ and consider for  $i\geq 1$, $\U_{i+1}=\U_i^p[\U_i,\U]$
the $p$-central descending series of $\U$. (For $p=2$, take $\U_{i+1}=\U_i^4[\U_i,\U_i]$.)

\begin{prop}
 Let $p$ be an odd prime, and $\U$  a uniform  pro-$p$-group. Then for each $n$, $\U_n^{\mathrm{ab}}$ has rank $d$ and maps onto $(\Z/p^n\Z)^d$.
\end{prop}

\begin{proof}
 
 Take $n >1$. Let  $x \in \U_n$ be an element of a minimal family of generators of   $\U_n$: the element $x$ is not trivial in the quotient
  $\U_n/\U_n^p[\U_n,\U_n]$. As 
  $\U$ is uniform, one has $\U_n^p[\U_n,\U_n]=\U_{n+1}$ and then $x$ is not trivial in $\U_n/\U_{n+1}$.
 Suppose now that the order $p^k$  of $x$ in $\U_n/\U_{n+1}$ is smaller than $p^{n-1}$, {\it i.e.} $x^{p^k} \in [\U_n,\U_n] $
 with $k <n$. Then as  $[\U_n,\U_n]\subset \U_{2n}$, one has $x^{p^k} \in \U_{2n}$.
 But as  $\U$ is uniform, for all  $m$ the following isomorphism holds: $$\U_{n}/\U_{n+1} \stackrel{x\mapsto x^{p^m}}{\longrightarrow} \U_{n+m}/\U_{n+m+1}.$$
The integer  $k$ being supposed smaller than  $n$, we find $x^{p^{n-1}}=1$ in $\U_{2n-1}/\U_{2n}$ and then 
$x =1$  in  $\U_n/\U_{n+1}$. Contradiction. Hence every 
 element of a generator basis of  $\U_n$ is of order at least $p^{n}$.
\end{proof}

\begin{coro}\label{analytic}
 Let $\G$ be a  uniform analytic pro-$p$-group of dimension $d$. Consider  the sequence 
 $\Mo_{\G_n^{\mathrm{ab}}}$ of mean exponents for the abelianizations of terms of the $p$-central
 series.
 We have $$\Mo_{\G_n^{\mathrm{ab}}} \geq n =\frac{1}{d}\log_p [\G:\G_n].$$
\end{coro}
\begin{proof}
Follows immediately from the previous Proposition.
\end{proof}

\begin{rema}[\cite{DSMN}, Chapter 13] 
Let us replace   $\Z_p$ by the  complete local regular Noetherien  ring $R=\Z_p[[T_1,\cdots, T_k]]$ with residue field $\fq_p$ and dimension $k+1$; here $\mm=(p,T_1,\cdots, T_k)$ is the maximal ideal of $R$.
 Let $\displaystyle{\grad(R) =\bigoplus_{i\geq 0} \mm^i/\mm^{i+1}}$ be the  graded algebra; put $c_i = \dim_{\fq_p} \mm^i/\mm^{i+1}$. 
 Following the terminology of  \cite{DSMN}, 
 consider $\G$ an $R$ standard and perfect group of dimension $d$. For example ${\rm Sl}_n^1(R):= \ker \left( {\rm Sl}_n(R) \rightarrow  {\rm Sl}_n(\fq_p)\right)$ is such a group for  $p>2$. In particular, $\G=\mm^d$ as analytic variety on which there is a formal group law $F$. Let us consider the filtration of $\G$: $\G_n\simeq (\mm^n)^d$, $n\geq 1$.
 Then, for all integers  $m,n\geq 1$, $[\G_m,\G_n]= \G_{m+n}$ ($\G$ is perfect) and there is an isomorphism of groups $\G_n^{ab}\simeq \left(\mm^n/\mm^{2n}\right)^d$, where the formal law on the quotient $\mm^n/\mm^{2n}$ becomes the addition.
 As the quotients $\mm^i/\mm^{i+1}$ are $p$-elementary, one has $$v_p([\G:\G_n])=\log_p [R:\mm^n]=c_1+c_2+\cdots + c_{n-1}.$$
 By using the  Hilbert-Samuel polynomial $H = CX^{k+1}+ \cdots$ of $\grad(R)$, $C>0$ ({\it i.e.}  
 $\deg(H) = k+1$), we have
 $$v_p([\G:\G_{n}]) \sim_n d H(n-1) \sim_n C d n^{k+1},$$ and
  $$v_p(|\G_n^{ab}|)=v_p([\G_{n}:\G_{2n}]) \sim_n d\left(H(2n-1)-H(n-1)\right) \sim_n c d (k+1) n^{k+1}(2^{k+1}-1).$$
  To finish, we want to bound the $p$-rank $d_p \G_n$ of $\G_n$:  $d_p \G_n =d\cdot  d_p (\mm^n/(p\mm^n+\mm^{2n}))$. First we have the following exact sequence:
  $$0 \longrightarrow  (p^{n-1}\mm+ \cdots  + p\mm^{n-1})/ p\mm^n \longrightarrow \mm^n/(p\mm^n+\mm^{2n}) \longrightarrow \mmm^n/\mmm^{2n} \longrightarrow 0,$$
  where $\mmm$ is the maximal ideal of $\fq_p[[T_1,\cdots, T_k]]$.
  Now the natural morphism:
  $$ \begin{array}{rcl}\mmm/\mmm^2 \times \cdots \times \mmm^{n-1}/\mmm^n & \rightarrow &  p^{n-1}\mm+ \cdots  + p\mm^{n-1} \ \bmod  p\mm^n \\
  ({\overline x_1},\cdots, {\overline x_{n-1}}) & \mapsto & p^{n-1} x_1 + \cdots p x_{n-1} \ \bmod p\mm^n
  \end{array}$$ allows us to obtain  $$d_p \G_n \leq a_1 + \cdots + a_{2n-1},$$
  where $a_i= d_p \mmm^{i-1}/\mmm^{i}$. The local ring $\fq_p[[T_1,\cdots, T_k]]$ is of dimension $k$, and then, if $\overline{H}=C' X^{k} + \cdots $
  is the Hilbert-Samuel of the graded algebra $\fq_p[[T_1,\cdots, T_k]]$, we have for $n\gg 0$:
  $$d_p \G_n \ll n^{k}.$$ 
  Finally, one obtains $$ \Mo_{\G_n^{\mathrm{ab}}}  \gg n\gg \left(\log_p[\G:\G_n]\right)^{1/(k+1)} .$$
\end{rema}

\subsection{Bounding $\Minf(\G_S^T)$ for tame $S$}

First, thanks to Proposition \ref{proposition2.15}, for the Galois group  $\G=\G_S^T$  of a tame tower $\K_S^T/\K$,  we have $$\Minf(\G) \leq c(\K,S,T) \limsup_\U \frac{[\G:\U]}{d(\U)},$$ 
where $c(\K,S,T)$ is a quantity that depends only on $\K,S,T$.  So, we must consider the rate of growth of the generator rank of open subgroups of $\G$ with respect to their index.  Recall that the \emph{rank gradient} of $\G$ (see, for example \cite{Ershov}) is defined to be
$$
\rho(\G) = \liminf_{\H} \frac{d(\H) - 1}{[\G:\H]},
$$
where the infimum is taken over all open subgroups $\H \subset \G$.
Note that when $\U\subset \V$,  Schreier's formula gives the inequality  
$\displaystyle{\frac{d(\U)-1}{[\G:\U]} \leq \frac{d(\V)-1}{[\G:\V]}}$ showing that  the sequence $[\G:\U_i]/d(\U_i)$ is increasing for a nested sequence $(\U_i)$ of open subgroups.  For groups with positive rank gradient $\varepsilon$, the $p$-rank of open subgroups grows $\varepsilon$-lineraly with the index (compare definition  \ref{croissancelineaire}).

\

In the general case, lacking any knowledge on the behavior of $d(\U)$, we nonetheless have the following result (Part 1 of Theorem \ref{intro}).

\begin{prop}\label{theo-1}
Suppose $S$ is a finite set of primes of a number field $\K$ with $(S,p)=1$. Let $\G=\G_S^T$.  There is a constant $C>0$ such that for any open subgroup $\U$ of $\G$, we have
$\Mo_{\U^{\mathrm{ab}}} \leq C [\G:\U].$
\end{prop}
\begin{proof}
We simply apply Proposition \ref{proposition2.15}, merely noting that $d(\U)\geq 1$.
\end{proof}

\begin{Question}
Is the conclusion of Proposition \ref{theo-1} true for every \emph{FAb} pro-$p$ group of finite type?
\end{Question}


\medskip

In the  main result of this section, for certain special subgroups $\U$ of $\G$, we give lower bounds  for $d(\U)$,  which allows us to estimate $\Mo_{\U^{\mathrm{ab}}}$.  The main references are \S 11 and \S 12 of \cite{DSMN}.

\

First of all, a key result is a Theorem of Jennings which asserts that for any group there exists a connection between the enveloping algebra associated to a certain graduated algebra $\grad(\G)$ of $\G$ and the restricted enveloping algebra of $\fq_p[G]$ graded by the ideal augmentation $I$. Here, $\grad(\G):=\oplus_{i\geq 0} D_i/D_{i+1}$, where $D_i=(1+ I^n) \cap \G$; put $b_i:=d_p D_i/D_{i+1}$.  The filtration $(D_n)$ is called the Zassenhaus filtration of $\G$; this filtration satisfies these  mains properties: $$D_1=G, \ D_n=D_{n^*}^p\prod_{i+j=n}[D_i,D_j], \  D_n^p \subset D_{np}, \ {\rm and} \ [D_n,D_m] \subset D_{n+m},$$
where $n^*=\lceil n/p \rceil$. Hence, $D_i/D_{i+1}\simeq (\Z/p\Z)^{b_i}$.

\medskip

The relationship between these two associative algebras gives a link between the $b_i $ and the $c_j:=d_p I^j/I^{j+1}$. More precisely, if $\displaystyle{U(T):=\sum_{ n\geq 0} c_n T^n}$ is the Hilbert Poincar\'e series of the graded algebra $\fq_p[[\G ]]$ then 
$$U(T)= \prod_{i\geq 1}\left(\frac{T^{pi}-1}{T^i-1}\right)^{b_i}.$$ In particular, when $\G$ is analytic the $p$-rank of its open subgroups are bounded and then, the integers $b_i$ should often vanish.  In fact, one has the spectacular result that $b_i=0$ for a single integer $i$ if and only if the pro-$p$-group is analytic. The following beautiful  lemma is a consequence of all of this:

\begin{lemm} \label{minorationrangbis} Suppose $\varepsilon >0$. If $\G$ is not analytic, then there exist infinitely many $n$ such that $$d_p D_{2^n} \geq (1-\varepsilon) \log_p [\G:D_{2^n}],$$
where  $D_{2^n}$ runs in  the Zassenhaus filtration $(D_k)$ of $\G$ 
\end{lemm}
 
 \begin{proof} It is the lemma 11.8 of \cite{DSMN}.
  \end{proof}

    \begin{defi} 
  A finitely generated pro-$p$ group $\G$ is said to be of Golod-Shafarevich type if all the relations
  are of degree $2$ and 
 $d^2\geq 4r$ where $d,r$ are the generator rank and relation rank of $\G$, respectively, cf.~Theorem \ref{GS}.
\end{defi}

\begin{rema}
A pro-$p$-group of Golod-Shafarevich type with relation rank $r>1$ is not analytic, 
cf.~\cite{lazard} and \cite{Serre}.
 If a pro-$p$ group is mild, with respect to the Zassenhaus filtration, and
all its relations are of degree $2$, then it is of Golod-Shafarevich type (and of cohomological dimension $2$) -- see \cite{labute}.
\end{rema}

\begin{prop} Suppose that the conditions of Theorem \ref{theocritere} hold for a number field $\K$,
so that $\G=\G_\emptyset^T$ is infinite.
 Then there exists a constant $C$ and infinitely many $n$ such that,  $$\Mo_{D_{2^n}^{\mathrm{ab}}} \leq C \frac{[\G: D_{2^n}]}{\log_p[\G:D_{2^n}]},$$
 where  $D_{2^n}$ runs in  the Zassenhaus filtration $(D_k)$ of $\G$.
\end{prop}
\begin{proof}
The conditions of Theorem \ref{theocritere} entail that $\G$ is of Golod-Shafarevich type, hence is not
analytic.  The desired conclusion is
therefore a consequence of Lemma \ref{minorationrangbis} and Proposition \ref{proposition2.15}.
\end{proof}

\

  To finish, let us  improve the lower bound of Lemma \ref{minorationrangbis}.
To simplify, assume that $p>2$.
  
   Let $$ 1 \longrightarrow R \longrightarrow F \longrightarrow \G \longrightarrow 1,$$
be a minimal presentation of $\G$: the pro-$p$-group $F$ is free and generated by $d$ elements  $x_1,\cdots, x_d$.
  We assume that $\G$ is finitely presented: the dimension over $\fq_p$ of $H^2(\G,\fq_p)$ is finite. Let  $\rho_1, \cdots, \rho_r \ \in F$ be
   a system of generators of $R/R^p[F,R]$. For $i=1,\cdots, r$, let $a_i$ be the degree of $\rho_i$ following the Zassenhaus filtration of $F$.

\begin{defi} For two formal series with real coefficients, 
we say that \ $\sum_n \alpha_n T^n \geq  \sum_n \alpha_n' T^n$ \ if for all $n$, $\alpha_n \geq \alpha_n'$.
\end{defi}

  \begin{prop}
Let $\G$ be a finitely presented pro-$p$-group.
Let $U(t)$ be the Hilbert Poincar\'e series of the graded algebra $\fq_p[[\G]]$. Then $$U(T) \geq \frac{1}{1-dT+\sum_{i=1}^r T^{a_i}},$$
with equality if $\G$ is of cohomological dimension at most $2$.
\end{prop}
  \begin{proof} The proof is essentialy a result of Brumer \cite{Brumer}.
First let us consider the natural short exact sequence
$$0 \longrightarrow \I(\G) \longrightarrow \fq_p[[\G]] \longrightarrow \fq_p \longrightarrow 0,$$
where $\I(\G)$ is the augmentation ideal of the complete algebra $\fq_p[[\G]]$.
The topological generators of $\G$ are in $\I(\G)$ and therefore all of degree $1$.
For a minimal presentation $$ 1 \longrightarrow R \longrightarrow F \longrightarrow \G \longrightarrow 1,$$ of $\G$, Brumer (see (5.2.1) in \cite{Brumer}), shows that there is a short exact sequence
$$0 \longrightarrow R/R^p[R,R] \stackrel{f}{\longrightarrow} \I(F)/\I(F)\I(R) \stackrel{g}{\longrightarrow} \I(\G) \longrightarrow 0,$$
where $f(r)=r-1 \bmod \I(F)\I(R)$.
Now, the quotient $\I(F)/\I(F)\I(R)$ is a free $\fq_p[[\G]]$-module on the generators $x_1-1,\cdots, x_d-1$ and then 
we have the relation on the Hilbert Poincar\'e series:
$$P(T)-d T U(T) + U(T) -1 =0,$$
where $P(T)$ is the series of $R/R^p[R,R]$ and where $U(T)$ is the series of $\fq_p[[\G]]$.
As $\fq_p[[\G]]\cdot \rho _1 \oplus \cdots \oplus \fq_p[[\G]] \cdot \rho_r \stackrel{\varphi}{\twoheadrightarrow} R/R^p[R,R]$, and that the elements $\rho_i$ are of degree $a_i$, one has
$$\displaystyle{P(T) \leq \left(\sum_{i=1}^r T^{a_i}\right) U(T)}.$$

Now, the equality comes from the fact that the pro-$p$-group $\G$ is of cohomological dimension  at most $2$ if and only if the map $\varphi$ is an isomorphism (see Proposition 5.3 in \cite{Brumer}).
\end{proof}


\begin{theo} \label{theo-2}
Let $\L/\K$ be a tamely ramified pro-$p$-extension with Galois group $\G$.  
Suppose that $\G$ is of Golod-Shafarevich type and of cohomological dimension $2$. Then for every $\varepsilon >0$,  there exists a constant $C$ and infinitely many $n$ such that 
$$\displaystyle{\Mo_{D_{2^n}^{\mathrm{ab}} } \leq C \frac{[\G:D_{2^n}]}{(\log_p[\G:D_{2^n}])^{2-\varepsilon}}},$$
where  $D_{2^n}$ runs in  the Zassenhaus filtration $(D_k)$ of $\G$.
\end{theo}

\begin{rema} In the inequality of the previous Theorem, the constant depends on 
 $\varepsilon$ and on the set of primes ramifying in $\L/\K$.
We note that Labute (Theorem 1.6 of \cite{labute}) was the first to give a sufficient condition for mildness of $\G_S^T$; thanks to the work of Schmidt \cite{schmidt}, for any $\K$, by choosing $S$ large enough, one
 can arrange that the group $\G_S^T$ is of cohomological dimension $2$ and 
 mild, hence meets the conditions of the Theorem.  (See also the work of Labute \cite{labute},  Forr\'e \cite{Forre}, G\"artner \cite{Gartner}, Vogel \cite{Vogel}, etc.)
We wish to highlight the fact that the preceding Theorem combines together some results from  analytic number theory (Brauer-Siegel), arithmetic (the results of Schmidt and the fact that the  root discriminant is bounded) and group theory! 
\end{rema}

\begin{proof}

We want to give a lower bound of $d_p D_{2^n}$. First, As $[D_{2^n},D_{2^n}] \subset D_{2^{n+1}}$, we should have in mind the fact that
$d_p D_{2^n} \geq d_p D_{2^n}/D_{2^{n+1}}$.

Now by hypothesis $$\prod_{i\geq 1}\left(\frac{T^{pi}-1}{T^i-1}\right)^{b_i} = \frac{1}{1-dT+rT^2} = \frac{1}{(1-\alpha T)(1-\beta T)},$$
with $\alpha \geq \beta$, $\alpha \geq 2$ and $\beta >1$. Indeed, as $\G^{ab}$ is finite,  $r\geq d$.

By taking logarithms, one obtains:
$$\sum_{i\geq 1} b_i \sum_{k\geq 1} \frac{1}{k}\left( T^{ki}- T^{pki}\right) = \sum_{i \geq 1} \frac{1}{i} (\alpha^i + \beta^i) T^i.$$
Take  $m$ with  $(m,p)=1$. Then by looking the coefficients at $T^m$:
$$\alpha^m +\beta^m= \sum_{i | m} i b_i .$$
This equality at  $m=2^n$ and at $m=2^{n-1}$ allows us to give:
$$b_{2^n}=2^{-n}\left(\alpha^{2^n}-\sqrt{\alpha^{2^n}}+\beta^{2^n}-\sqrt{\beta^{2^n}}\right)$$ 
and then there is a  constant $C>1$ such that for all large enough $n$, we have:  $$b_{2^n} \geq C \frac{\alpha^{2^n}}{2^n}.$$

\medskip

Let us conserve the notation of \cite{DSMN} and  put $i_n=\log_p [\G:D_{2^n}]$.
As $d_p D_{2^n} \geq d_p D_{2^n}/D_{2^{n+1}} = \log_p |D_{2^n}/D_{2^{n+1}}|$ 
one has the inequality $$i_{n+1} \leq  d_n + i_n ,$$
where $d_n = d_p D_{2^n}$. Now, for $n\gg 0$,   $$i_{n+1} = \log_p[\G:D_{2^{n+1}}] = \log_p[\G:D_{2^n}] + \log_p[D_{2^n}:D_{2^n+1}] + \log_p[D_{2^n+1}:D_{2^{n+1}}] \geq  b_{2^n} \geq C \frac{\alpha^{2^n}}{2^n}.$$ 
Let $n_0$ be an integer. Suppose that for all  $n\geq n_0$, \ $d_n \leq i_n^{2-\varepsilon}$.  Then, $i_{n+1} \leq 2 i_n^{2-\varepsilon}$ and by induction $$i_{n+1} \leq 2^{1+(2-\varepsilon)+\cdots + (2-\varepsilon)^{n-n_0}} i_{n_0}^{(2-\varepsilon)^{n+1-n_0}}.$$ Hence for $n\gg n_0$, 
$$C \frac{\alpha^{2^n}}{2^n} \leq i_{n+1} \leq 2^{\frac{(2-\varepsilon)^{n+1-n_0}-1}{1-\varepsilon}} i_{n_0}^{(2-\varepsilon)^{n+1-n_0}}$$ which is a contradiction for large $n$.

Hence, there exist infinitely many $n$ such that $d_p D_{2^n} \geq (\log_p[\G:D_{2^n}])^{2-\varepsilon}$ and if $\G$ is the Galois group of tamely ramified tower,
$\displaystyle{\Mo_{D_{2^n}^{\mathrm{ab}}} \ll \frac{[\G:D_{2^n}]}{(\log_p[\G:D_{2^n}])^{2-\varepsilon}}}$.
\end{proof}

\begin{rema}
Calculations of the above type with Poincar\'e series can be found, for example, in \cite{mcleman}
and \cite{minac}.
\end{rema}

 \section{Final Remarks}

 \subsection{On a question of structure}
 We have been looking for towers in which the $p$-rank of 
class groups has 
 linear growth. In the Iwasawa context, abelian as well as non-abelian
 (for the latter see for example \cite{Perbet}), there is an
 underlying algebraic structure thanks to which linear growth of the rank
 corresponds exactly to having positive $\mu$-invariant.  Can we
 detect any evidence of a similar algebraic structure in the tame
 case?

 \
 
 In this paper we produce our examples as follows. First, we consider
 an infinite extension $\k_S^T/\k$ with $T$ non-trivial, and then take
its compositum with a finite $p$-extension $\K/\k$ inside $\k_T$.  In 
this manner, one 
obtains a subextension $\L:=\K\k_S^T$ of
 $\k_{\{S\cup T\}}^\emptyset$.  It is in the extension $\L/\K$ that we can force
linear growth of the $p$-class groups $(\A_n)_n$.  Put
 $\G=\Gal(\k_S^T/\k)\simeq \Gal(\L/\K)$. By a result of Schmidt 
\cite{schmidt}, by choosing $S$ large enough, one
 can assume that the group $\G$ is of cohomological dimension $2$ and
 mild.  Let $\Lambda:=\fq_p[[\G]]$ be the Iwasawa algebra associated
 to $\G$.  As $\G$ is mild, the ring $\Lambda$ is without zero
 divisor, but note that it's probably not Noetherian.  
Let $\displaystyle{\X:=\lim_{\leftarrow_n} \A_{n} }$
 be the projective limit of the studied arithmetic object $\A_n$. The
 limit $\X$ is a finitely generated $\Lambda$-module
 (\cite{Maire-MZ}).
 \begin{Question} Is the linear growth of $A_n$ produced by this method 
related to  a natural algebraic structure of ``Iwasawa module''
$\X$ ?
 \end{Question}

\subsection{How small can the mean exponent be in tame towers?}
 
 We have shown that there exist asymptotically good infinite towers in which the mean
 exponent is bounded above.  On the one hand, it is natural to wonder:
 \begin{Question}\label{q2}
Can we find asymptotically good pro-$p$ towers  $\LL$ 
for which $\Minf(\LL)$ is arbitrarily close to~$1$? 
\end{Question}\label{q1}
 On the other hand, our constructions are rather special, so we ask:
\begin{Question}
Are there asymptotically good infinite pro-$p$ towers in which the mean exponent of $p$-class groups is not bounded?
\end{Question}
 
 As a start on Question \ref{q2}, we note that in section \ref{exemples}, we have developed some examples of the
 following type: $$\K=\Q(\sqrt{p_1\cdots p_t},\sqrt{- p_{t+1}\cdots
   p_{t+s}}).$$ Here $\k^T/\k$ is infinite where
 $\k=\Q(\sqrt{p_1\cdots p_t})$ and $T=\{ p_{t+1},\cdots ,p_{t+s}\}$.
 These examples give $s$-linear growth for $p$-class groups
 where the base field $\K$ has genus $g \approx
 \log(p_1 \cdots p_t p_{t+1} \cdots p_{t+s})$.  Letting $n=t+s$, we
note that as $n$ becomes
 large, one has $g \lesssim p_n$,
 where $p_n$ is, in the optimal case,
 the $n$th prime number {\it i.e.} $g \sim n \log(n)$. But
 on the other side, to force the infinitude of $\k^T/\k$, which we need,
 we must apply corollary \ref{critere}, which requires $s \sim n$. Thus,
the best we can expect via this method for bounding $\Minf(\K_\emptyset^\emptyset/\K)$ is only
$\Minf(\K_\emptyset^\emptyset/\K) \lesssim \log(n)$.
 
\begin{Question}
 What is the  biquadratic field  (following the above method) with the smallest 
 upper bound on the value of $\Minf(\K_\emptyset^\emptyset/\K)$?
\end{Question}

\subsection{Concluding Summary} In this paper, we have introduced the logarithmic mean exponent of a finite abelian $p$-group as an invariant 
that balances the cardinality  of the group as weighed against its rank, and studied 
its behavior in the context of $p$-class groups of number fields varying in towers with restricted 
ramification.  By a mixture of results from algebraic and analytic number theory, we have constructed tame towers for which the mean exponent is bounded, and shown that, by contrast, the mean exponent for some open subgroups of $p$-adic analytic groups tend to infinity.  We hope that further study of the mean exponent will shed light on properties that distinguish Galois groups of tame versus wild extensions.



\begin{thebibliography}{99}
\bibitem{Brumer}A. Brumer,  {\em Pseudocompact Algebras, Profinite Groups and Class Formations}, J. Algebra {\bf 4} (1966), 442-470.
\bibitem{CSJ} J. Coates, P. Schneider and R. Sujatha, {\em Modules over Iwasawa algebras}, J. Math. Inst. Jussieu {\bf 2}:1 (2003), 73-108. 
\bibitem{DSMN} J.D. Dixon, M.P.F. Du Sautoy, A. Mann  and D.Segal, {\em Analytic pro-$p$-groups}, Cambridge studies in advances mathematics 61, Cambridge University Press, 1999.
\bibitem{Ershov}M. Ershov, {\em Golod-Shafarevich groups: a survey}. Internat. J. Algebra Comput. {\bf 22} (2012), no. 5, 68pp.
\bibitem{FM}J.-M. Fontaine and B. Mazur, {\em Geometric Galois representations},
In Elliptic curves, modular forms, and Fermat's last theorem 
(Hong Kong, 1993), 41--78, Ser. Number Theory, I, Internat. Press, 
Cambridge, MA, 1995. 
\bibitem{Forre} P. Forr\'e, {\em Strongly free sequences and pro-$p$-groups of cohomological dimension $2$}, {J. reine angew. Math.} {\bf 658}
(2011), 173-192.
\bibitem{Friedman}E. Friedman, {\em Analytic formulas for the regulator of a number field}. Invent. Math. {\bf 98} (1989), no. 3, 599-622.
\bibitem{Gartner} J. G\"artner, {\em Mild pro-p-groups with trivial cup-product}, PhD, University of Heildeberg, 2011.
\bibitem{Gartner-1} J. G\"artner, {\em On p-class groups and the Fontaine-Mazur Conjecture }, Math. Research Letters {\bf 21} (2014), 469-477.
\bibitem{gras}G. Gras, Class Field Theory, SMM, Springer, 2003.
\bibitem{Gras-Munnier}G. Gras and A. Munnier, {\em Extensions cycliques T-totalement ramifi\'ees}, Publ. Math. Besan\c con, 1997/98.
 \bibitem{Hajir}Hajir F., {\em On the growth of $p$-class groups in p-class field towers}, J. Algebra {\bf 188} (1997),
 no. 1, 256-271.
 \bibitem{HMcompositio} F. Hajir and C. Maire, {\em Tamely ramified towers and discriminant bounds for number fields}, Compositio Math. {\bf 128} (2001), no. 1, 35-53.
 \bibitem{HM-IMRN}F. Hajir and C. Maire, {\em Extensions of number fields with bounded ramification of bounded depth}, IMRN {\bf 13} 2002, 677--696.
 \bibitem{Harris1} M. Harris, {\em $p$-adic representations arising from descent on abelian varieties}, Compositio Math. {\bf 39}:2 (1979), 177-245. With Correction: Compositio Math. {\bf 121}:1 (2000), 105-108.
 \bibitem{Ihara}Y. Ihara,  {\em How many primes decompose completely in an infinite unramified Galois extension of a global field?}, 
 J. Math. Soc. Japan {\bf 35} (1983), no. 4, 693-709.
 \bibitem{Iw}K.~Iwasawa, {\em On the $\mu$-invariants of $\Z_\ell$-extensions}. Number theory, algebraic geometry and commutative algebra, in honor of Yasuo Akizuki, pp. 1--11. Kinokuniya, Tokyo, 1973. 
 \bibitem{koch}H. Koch, Galoissche Theorie der $p$-Erweiterungen,  Springer-Verlag, 1970.
 \bibitem{labute}J. Labute, {\em Mild pro-$p$-groups and Galois groups of $p$-extensions of $\Q$}, J. reine angew. Math., {\bf 596} (2006), 155-182.
 \bibitem{Lang} S. Lang, Algebraic Number Theory, Graduate Texts in Mathematics, Springer-Verlage, New York/Berlin, 1986.
 \bibitem{lazard} M. Lazard, {\em Groupes analytiques $p$-adiques}, IHES, Publ. Math. {\bf 26} (1965), 389-603.
\bibitem{Maire-Bx}C. Maire,  {\em Finitude de tours et $p$-tours $T$-ramifi\'ees mod\'er\'ees, $S$-d\'ecompos\'ees}, J. Th. des Nombres 
de Bordeaux {\bf 8} (1996), 47-73.
\bibitem{Maire-PMB}C. Maire, {\em $T$-$S$- Capitulation},  Publ. Math. Fac. Sci. Besançon, 1994-1995.
\bibitem{Maire-MZ} C. Maire, {\em Sur la structure galoisienne de certaines pro-$p$-extensions de corps de nombres},  Math. Z. {\bf 267} (2011), no. 3-4, 887-913.
\bibitem{mcleman}M. McLeman, {A Golod-Shafarevich equality and $p$-tower groups,} J. Number Theory {\bf 129} (2009), no.~11, 2808, 2819.
\bibitem{minac} Jan Minac, Michael Rogelstad, and Nguyen Duy Tan, {\em
Dimensions of Zassenhaus filtration subquotients of some pro-$p$-groups}, preprint, 2015, 25pp. arXiv:1405.6980.
\bibitem{NSW}J. Neukirch, A. Schmidt, K. Wingberg, Cohomology of Number Fields, GMW 323, Springer 2008.
\bibitem{Ozaki}M. Ozaki, {\em Construction of maximal unramified p-extensions with prescribed Galois groups} Invent. Math. {\bf 183} (2011), no. 3, 649-680. 
\bibitem{Perbet} G. Perbet, {\em Sur les invariants d'Iwasawa dans les extensions de Lie p-adiques} (French) [On Iwasawa invariants in p-adic Lie extensions] Algebra Number Theory {\bf 5} (2011), no. 6, 819-848.
\bibitem{schmidt}A. Schmidt, {\em \"Uber pro-$p$-fundamentalgruppen markierter arithmetischer kurven, } J. Reine Angew. Math. {\bf 640} (2010), 203--235. 
\bibitem{Serre}J.-P. Serre, Cohomologie Galoisienne, Lecture Notes in Mathematics 5, Springer-Verlag, 1984. 
\bibitem{TV}M. Tsfasman and S. Vladut, {\em Infinite global fields and the generalized Brauer-Siegel theorem}. 
Dedicated to Yuri I. Manin on the occasion of his 65th birthday. Mosc. Math. J.{\bf 2} (2002), no. 2, 329-402.
\bibitem{Vogel}D. Vogel, {\em Massey products in the Galois cohomology of number fields}, PhD Heidelberg, 2004.
\bibitem{Venjakob}O. Venjakov, {\em On the structure theory of the Iwasawa algebra of a $p$-adic Lie group}, J. Eur. Math. Soc.{\bf 4} (2002), 271-311.
\bibitem{Zimmert} R. Zimmert, {\em Ideale kleiner Norm in Idealklasse une eine Regulatorabsch\"atzung}, Invent. Math. {\bf 62} (1981), no 3, 367-380.

\end{thebibliography}
\end{document}